\documentclass{article}
\usepackage[round,authoryear]{natbib}
\usepackage[margin=1in]{geometry}
\usepackage[T1]{fontenc}
\usepackage{amsmath,amsthm,mathtools}
\usepackage{newtxtext,newtxmath}
\usepackage{microtype}
\usepackage{enumitem}
\usepackage{xcolor}

\usepackage[colorlinks=true,linkcolor=blue!45!black,
  citecolor=blue!45!black,urlcolor=blue!45!black]{hyperref}

\newtheorem{theorem}{Theorem}[section]
\newtheorem{lemma}[theorem]{Lemma}
\newtheorem{proposition}[theorem]{Proposition}
\newtheorem{corollary}[theorem]{Corollary}
\newtheoremstyle{italicremark}
  {\topsep}
  {\topsep}
  {\itshape}
  {}
  {\itshape}
  {.}
  {.5em}
  {}
\theoremstyle{italicremark}
\newtheorem{remark}[theorem]{Remark}

\newcommand{\R}{\mathbb{R}}
\newcommand{\cC}{\mathcal{C}}
\newcommand{\cF}{\mathcal{F}}
\newcommand{\cM}{\mathcal{M}}
\newcommand{\argmin}{\operatorname*{argmin}}
\newcommand{\dist}{\operatorname{dist}}

\setlist[itemize]{leftmargin=1.7em,itemsep=0.15em,topsep=0.3em}
\setlist[enumerate]{leftmargin=1.9em,itemsep=0.2em,topsep=0.3em}
\allowdisplaybreaks

\title{A lower bound for stepsize-based acceleration of gradient descent}
\author{Jianhao Ma\thanks{Department of Industrial Engineering, Tsinghua University. Email:
\texttt{jianhao@tsinghua.edu.cn}.}
\and Yuxin Chen\thanks{Department of Statistics and Data Science, Wharton School,
University of Pennsylvania. Email:
\texttt{yuxinc@wharton.upenn.edu}.}}
\date{\today}

\begin{document}
\maketitle

\begin{abstract}
Recent work has shown that, for smooth convex optimization, plain gradient descent  can be accelerated from its textbook convergence rate of $O(T^{-1})$ (where $T$ denotes the number of iterations) to $O\big(T^{-\log_2(1+\sqrt{2})}\big)$  using carefully designed stepsize schedules alone, without resorting to momentum or other algorithmic modifications. Despite this progress, however, little was known about lower bounds for such methods beyond the classical $\Omega(T^{-2})$ benchmark for general first-order methods. In this work, we present a new lower bound of $\Omega(T^{-1.9319})$ for the last-iterate convergence rate of gradient descent with predetermined nonnegative stepsize schedules.  This result provides rigorous evidence that stepsize schedules alone cannot accelerate plain GD to the optimal $O(T^{-2})$ convergence rate. The proof was developed by GPT-5.6 Sol Pro under the authors' guidance. 
\end{abstract}

\section{Introduction}

Consider the unconstrained smooth convex optimization problem
\begin{align}
\text{minimize}_{x\in\mathbb{R}^d}~~~ f(x),
\label{eq:min-convex-problem}
\end{align}
where $f:\mathbb{R}^d\to\mathbb{R}$ is a convex function whose gradient is $L$-Lipschitz.
 Gradient descent (GD), a canonical first-order method for solving \eqref{eq:min-convex-problem}, generates a sequence of iterates according to
\begin{align}
x_{t+1}
=
x_t-\eta_t\nabla f(x_t),
\qquad
0\leq t<T,
\label{eq:gd-intro}
\end{align}
where $x_0$ is the initialization, $T$ is the prescribed number of iterations, and $\{\eta_t\}_{0\leq t< T}$ is the stepsize schedule. With the standard constant stepsize $\eta_t=1/L$, the suboptimality of its last iterate decreases at the rate $O(T^{-1})$  \citep{bubeck2015convex}.  

\subsection{Motivation: stepsize-based acceleration of GD}

Faster convergence has traditionally required the use of momentum or other modifications to the GD update; most prominently, Nesterov's accelerated method attains the optimal $O(T^{-2})$ rate \citep{nesterov1983method}.  Surprisingly, a recent line of work has shown that plain GD can also be accelerated through carefully designed stepsize schedules alone --- e.g., schedules with occasional long steps --- even though such steps may temporarily worsen the objective value \citep{altschuler2025smoothconvex,grimmer2025objectivegap}.

The strongest known constructions are based on the
silver ratio $\rho_{\mathrm{sil}}=1+\sqrt2$.  Specifically, the stepsize schedules proposed by
\citet{altschuler2025smoothconvex,grimmer2025objectivegap} achieve the
last-iterate convergence rate of
\[
O\big(T^{-\log_2\rho_{\mathrm{sil}}}\big)
\qquad 
(\text{with }
\log_2\rho_{\mathrm{sil}} \approx 1.2715)
\]
whenever the {\em prescribed} number of iterations satisfies $T=2^k-1$ for some integer $k\ge1$.  
Subsequent concatenation and composition techniques extend the same convergence rate to arbitrary prescribed numbers of iterations \citep{zhang2024concatenation,grimmer2025composing}. 
These methods employ the standard GD update with predetermined nonnegative
stepsize schedules containing occasional long steps, chosen with knowledge of the total number of iterations $T$.  These developments naturally raise the following question:
\begin{quote}
\emph{How fast can plain GD converge when its stepsize schedule is chosen in advance with knowledge of the total number of iterations?}
\end{quote}
\citet{altschuler2025smoothconvex} conjectured that the optimal answer is given by the silver exponent $\log_2(1+\sqrt2)$, but this conjecture remains open.  

Resolving this question requires lower bounds tailored specifically to GD with predetermined stepsize schedules. 
The classical \(\Omega(T^{-2})\) lower bound for deterministic first-order methods applies to every method using at most \(T\) oracle calls and therefore does not distinguish plain GD from accelerated methods \citep{drori2016exact}.  Existing sharper lower bounds for GD with potentially large stepsizes either apply only to basic \(f\)-composable schedules \citep{grimmer2025composing}, or concern the ``anytime'' setting, where a single infinite stepsize schedule must be fixed in advance and perform well regardless of when the algorithm is stopped \citep{tsai2026lower}. In fact, existing lower bounds do not preclude plain GD with suitably designed predetermined stepsize schedules from attaining the optimal \(O(T^{-2})\) convergence rate.

\subsection{This paper}

\paragraph{Main contributions.} 
In this work, we make progress on this question for the class of
predetermined stepsize schedules designed for a prescribed horizon $T$, without
imposing any restrictions on the magnitudes and ordering of the nonnegative
stepsizes. 
Our main result establishes a new lower bound of 
$$
\Omega\big(T^{-p}\big)
\qquad \text{where  }p>p_\star=\sqrt{2+\sqrt3}\approx 1.9319
$$  
for the worst-case last-iterate convergence rate of gradient descent with arbitrary
predetermined nonnegative stepsize schedules. The proof constructs hard instances in
dimension at most $T+1$, and applies to schedules with zero or arbitrarily
large stepsizes, arbitrary ordering, and no descent or monotonicity
assumptions. To the best of our knowledge, this is the first rigorous evidence that adjusting predetermined stepsize schedules alone is insufficient to accelerate plain gradient descent to the optimal $O(T^{-2})$ convergence rate. The remaining gap between the best-known upper-bound
exponent, $\log_2(1+\sqrt2)\approx1.2715$, and our impossibility threshold
$p_\star$, remains unresolved.

\paragraph{Organization.}
The remainder of this paper is organized as follows. Section~\ref{sec:main-results} presents the main theorem, along with a high-level description of the proof strategy. 
Section~\ref{sec:related-work} reviews related work.
Sections~\ref{sec:geometric-construction},~\ref{sec:chronology}, and~\ref{sec:rank-analysis} provide the analysis (which
develops, respectively, the realization, the order-free matching
bound, and the final cutoff argument and scaling).   Section~\ref{sec:discussion}
concludes the paper with additional discussion, whereas the appendices contain 
the remaining technical analysis.

\paragraph{Notation.}
Throughout this paper, $d$ is a positive integer denoting the ambient dimension.  We equip $\R^d$
with its Euclidean inner product and norm $\|\cdot\|$.  We let $L>0$
denote the smoothness parameter, such that 
for all $x,y\in\R^d$,
\[
 \|\nabla f(x)-\nabla f(y)\|\le L\|x-y\|.
\]
Additionally, let
$\argmin f$
denote the set of minimizers of a function $f$. 
We also let $\operatorname{conv}(A)$ denote the convex hull of a set $A$.

\section{Main results}
\label{sec:main-results}

In this section, we formalize the setting and state our main lower bound. Throughout,
we consider gradient descent with predetermined nonnegative stepsize
schedules designed for a prescribed horizon $T$.

\subsection{Main theorem}

Let $\cF_{0,L}(\R^d)$ denote the class of differentiable convex functions
from $\R^d$ to $\R$ that admit at least one minimizer and have
$L$-Lipschitz continuous gradients. We consider the smooth convex optimization problem \eqref{eq:min-convex-problem}, where the objective function is assumed to obey $f\in\cF_{0,L}(\R^d)$. Given a prescribed horizon $T\ge1$ and an
initial point $x_0\in\R^d$, we consider running plain GD \eqref{eq:gd-intro} using a predetermined, nonnegative stepsize schedule
$\eta=(\eta_0,\ldots,\eta_{T-1})\in\R_{\ge0}^T$. 
The stepsize schedule is
fixed before the problem dimension, objective function, and minimizer are chosen. Our lower bound concerns the convergence rate of 
the last iterate $x_T$;  intermediate objective values need not decrease.

\begin{theorem}
\label{thm:main}
Let $p_\star \coloneqq \sqrt{2+\sqrt3}$.  For every
$p\in(p_\star,2)$, there exists a constant $c_p>0$, depending only on $p$,
such that the following holds.  For every integer $T\ge1$, every $L,R>0$,
and every predetermined nonnegative schedule
$\eta=(\eta_0,\ldots,\eta_{T-1})\in\R_{\ge0}^T$, there exists an integer
$d$ with $1\le d\le T+1$ such that, for every prescribed initial point
$x_0\in\R^d$, there exist a function $f\in\cF_{0,L}(\R^d)$ and a minimizer
$x_\star\in\argmin f$ satisfying $\|x_0-x_\star\|=R$ such that GD
(cf.~\eqref{eq:gd-intro}) initialized at $x_0$ with schedule $\eta$
satisfies
\[
 f(x_T)-f(x_\star)
 \ge c_p L R^2 (T+1)^{-p}.
\]
\end{theorem}

\begin{remark}
Theorem~\ref{thm:main} establishes the lower bound separately for every
prescribed horizon and every predetermined stepsize schedule; the
adversarial instance may depend on both. Additionally, it does not establish the
endpoint lower bound $\Omega(T^{-p_\star})$, corresponding to the case
$p=p_\star$. Instead, the theorem applies only to exponents
$p>p_\star$.
\end{remark}

\paragraph{Comparison with lower bounds for anytime schedules.}
Our result should be contrasted with the concurrent anytime lower bound
of \citet{tsai2026lower}, which rules out a worst-case last-iterate
convergence rate of $o(T^{-4/3})$ for any single infinite sequence of
strictly positive stepsizes. That result requires a single stepsize
schedule to provide convergence guarantees uniformly over all stopping
times, whereas Theorem~\ref{thm:main} allows a different predetermined
stepsize schedule for each prescribed horizon. Therefore, the anytime
lower bound therein does not imply our lower bound.

\subsection{Proof strategy}
\label{subsec:proof-strategy}

In this subsection, we give a roadmap for the proof of Theorem~\ref{thm:main}; the full
proof is provided in
Sections~\ref{sec:geometric-construction}-\ref{sec:rank-analysis}.
After normalizing the stepsizes as $h_t=L\eta_t$, we assume without loss of
generality that $L=R=1$. We decompose each normalized stepsize into its capped
part and its excess part, and define
\[
 y_t \coloneqq (h_t-1)_+,\qquad
 B \coloneqq 1+\sum_{t=0}^{T-1}\min\{h_t,1\},\qquad
 r \coloneqq \bigl|\{t:y_t>0\}\bigr|.
\]
We call a step \emph{long} if $h_t>1$, equivalently if $y_t>0$. Thus $B$ is the
base mass contributed by the unit-capped schedule, and $r$ is the number of long
steps. The proof proceeds in three stages. First, we construct an explicit hard
instance tailored to selected long steps in the prescribed schedule. Second, we
derive a lower bound on the suboptimality gap that depends only on the magnitudes
of these long steps and is independent of their temporal order. Third, we combine
a rank cutoff argument with a Lyapunov estimate to obtain the desired convergence
lower bound.

\paragraph{An explicit hard instance from selected long steps.}
Choose any $m$ long steps in their original temporal order. They divide the
schedule into $m+1$ blocks. Section~\ref{sec:geometric-construction} associates
with these blocks positive scales $H_0,\ldots,H_m$ and transition factors
$\chi_0,\ldots,\chi_{m-1}$. Here $H_i$ records the displacement accumulated in
block $i$, while $\chi_i$ is the largest squared-amplitude ratio for which the
projection comparisons in the construction remain valid.
The construction places orthogonal anchors
\[
 X_i=\lambda_i e_i,\qquad
 \lambda_0=1,\qquad
 \lambda_{i+1}^2=\chi_i\lambda_i^2,
\]
where $e_0,\ldots,e_m$ is an orthonormal basis of $\R^{m+1}$. The block
gradients $g_i$ are chosen so that the trajectory moves along the ray
$X_i-\R_{\ge0}g_i$ during block $i$, and the selected long step at the end of
the block sends the iterate exactly to the next anchor:
\[
 X_i-H_i g_i=X_{i+1}\qquad (0\le i<m).
\]
The terminal gradient is chosen analogously along the final coordinate. Let
$K:=\operatorname{conv}\{0,g_0,\ldots,g_m\}$,
and define the hard instance as the Moreau envelope of the support function of
$K$:
\[
F(x):=\min_{z\in\R^{m+1}}
 \left\{\sigma_K(z)+\frac12\|x-z\|^2\right\},
 \qquad
 \sigma_K(z):=\max_{g\in K}\langle g,z\rangle .
\]
The key property of this choice is the Moreau identity
$\nabla F(x)=\Pi_K(x)$. Projection comparisons show that the vertex $g_i$
remains the projection throughout block $i$. Hence the unselected steps move
along the ray $X_i-\R_{\ge0}g_i$, while the selected long step lands exactly at
the next orthogonal anchor. At the terminal time, the construction yields
\[
 F(x_T)-F(0)
 =\frac{\lambda_m^2}{2H_m}
 =\frac1{2H_m}\prod_{i=0}^{m-1}\chi_i.
\]
Maximizing $H_m^{-1}\prod_i\chi_i$ over all admissible choices leads to a key functional $\cC_T(h)$. The empty selection contributes
$(1+2\sum_t h_t)^{-1}$; in particular, it already yields an $\Omega(T^{-1})$
gap when all steps are short. The remaining task is to show that long steps
cannot make $\cC_T(h)$ too small.

\paragraph{Removing temporal order by two matchings.}
The main obstacle is that each factor $\chi_i$ couples two neighboring blocks,
so a chain value depends on the temporal order of the long steps, not only on
their sizes. To eliminate this temporal dependence, rank the positive excesses as
$a_1\ge\cdots\ge a_r>0$ and define the residual schedule mass
$D_q:=B+\sum_{s=q+1}^r a_s$.
For a fixed $2\le q\le r$, select the $q$ largest excesses while retaining their
original temporal order. A budget-equalization argument converts the reciprocal
of the corresponding chain value into the edge product of a path whose vertex
weights are normalized reciprocals of the selected excesses. The odd and even
edges of this path form two matchings. Replacing each by the corresponding
optimal matching value removes the temporal order completely, at the cost of
upper-bounding the reciprocal chain value. Therefore, after inversion, we obtain
a lower bound on $\cC_T(h)$. If $\mu_q$ denotes the resulting geometric mean
cost per edge, then
\[
 \cC_T(h)
 \ge \frac{q}{2D_q(q-1)\mu_q^q}.
\]
Thus a bound of the form $\mu_q<\rho<1$ immediately creates the favorable
factor $\rho^{-q}$ in the lower bound.

The single statistic
\[
 \zeta_q:=\frac{D_q}{q^2}\sum_{s=1}^q \frac1{a_s}
\]
controls the total normalized reciprocal weight of the selected excesses, and
therefore controls the matching cost $\mu_q$. A total-weight estimate for the
two matchings shows that, when $\zeta_q\le\vartheta$ and $q$ is large, $\mu_q$
is bounded above by $2\vartheta+2\vartheta^2$ plus an error that vanishes as $q\to\infty$. This is the only
property of the matching problems needed in the final rank argument.

\paragraph{Rank cutoff and mass growth.}
Fix $p\in(p_\star,2)$ and set $\vartheta:=(p^2-1)^{-1}$, which satisfies $2\vartheta+2\vartheta^2<1$. Thus, for all sufficiently
large ranks, one can choose a fixed $\rho<1$ such that
$\zeta_q\le\vartheta$ forces $\mu_q<\rho$. This produces a useful dichotomy.
If $\mu_q<\rho$, the matching bound already gives a large value of
$\cC_T(h)$ through the factor $\rho^{-q}$. Otherwise $\zeta_q>\vartheta$.
In this second regime, define
\[
 \nu_q:=\frac{qa_q}{D_q}.
\]
Since $D_{q-1}/D_q=1+\nu_q/q$, the quantity $\nu_q$ measures how quickly the
residual schedule mass grows as the rank cutoff decreases.

The pointwise bound on $\nu_q$ is not sufficient by itself. The decisive
additional input is the exact adjacent-rank recursion
\[
 \zeta_{q+1}
 =\frac{q^2\zeta_q}
 {(q+1)(q+1+\nu_{q+1})}
 +\frac1{(q+1)\nu_{q+1}},
 \qquad 1\le q<r.
\]
If $\nu_q$ were roughly constant and equal to $\nu$, this recursion would drive
$\zeta_q$ toward $1/(\nu(\nu+2))$. Maintaining $\zeta_q>\vartheta$ therefore
limits the sustainable growth rate to $\nu\le p-1$.

Section~\ref{sec:rank-analysis} makes this heuristic argument rigorous with the aid of a Lyapunov
potential. Telescoping its one-step drift and using the mass-ratio identity
yield a constant $K_p$, depending only on $p$, such that
\[
 D_k k^{p-1}\le K_p B r^{p-1}
\]
along every interval on which the second alternative persists.

The proof now scans the ranks until either the favorable matching alternative
occurs or the Lyapunov estimate controls all but a bounded number of
excesses; in the latter case, the remaining excesses are handled directly. In both cases,
\[
 \cC_T(h)
 \ge\frac{c_p}{B(r+1)^{p-1}}
 \ge c_p(T+1)^{-p}.
\]

\paragraph{The threshold exponent.}
The matching and mass-growth arguments impose competing requirements on
$\vartheta$: the Lyapunov analysis gives
$p=\sqrt{1+\vartheta^{-1}}$, whereas the matching estimate requires
\[
 2\vartheta+2\vartheta^2<1
 \quad\Longleftrightarrow\quad
 \vartheta<\frac{\sqrt3-1}{2}
 \quad\Longleftrightarrow\quad
 p>\sqrt{2+\sqrt3}=p_\star.
\]
At $p=p_\star$, the interval $(2\vartheta+2\vartheta^2,1)$ from which $\rho$
must be chosen collapses. This explains why Theorem~\ref{thm:main} establishes
$\Omega(T^{-p})$ for every $p\in(p_\star,2)$, but is not yet able to reach the endpoint
$p=p_\star$.

\section{Related work}
\label{sec:related-work}

\paragraph{Classical bounds and horizon-dependent analysis.}
For smooth convex optimization, classical worst-case theory has established an
$O(LR^2/T)$ convergence guarantee for plain GD with the
constant stepsize $\eta=1/L$
\citep{bubeck2015convex,beck2017first}, while accelerated first-order
methods achieve the optimal $O(LR^2/T^2)$ rate
\citep{nesterov1983method}.  In dimensions $d\ge T+1$, the
exact information-based minimax risk also scales on the order of $LR^2/T^2$ \citep{drori2016exact} and is
attained by the Optimized Gradient Method \citep{drori2014performance,kim2016optimized}. Since this
is an oracle lower bound---that is, it applies to the entire class of
first-order methods that access the objective only through a first-order
oracle---it characterizes the fundamental limits of first-order methods,
rather than those of plain GD itself.  A complementary line
of work studies lower bounds for restricted classes of
first-order methods. For instance, 
\citet{arjevani2016iteration} showed that restricting GD to 
time-invariant stepsizes cannot improve upon the classical $O(T^{-1})$
convergence rate, whereas our result allows arbitrary predetermined
stepsize schedules.

Another related line of work studies the exact worst-case
performance of first-order methods through the
performance estimation problem (PEP) framework, introduced by
\citet{drori2014performance}. Building on smooth interpolation,
\citet{taylor2017smooth} showed that, when $d\ge T+2$, the PEP can be
formulated as a finite-dimensional semidefinite program (SDP); an additional
rank constraint is required in smaller dimensions. 
For GD on $L$-smooth convex objectives,
\citet{daccache2019performance} derived the exact worst-case two-step objective-gap formula
for $(L\eta_0,L\eta_1)\in[0,1]^2$ and conjectured a formula on $[0,2]^2$,
together with partial extensions to longer schedules. \citet{diego2022worst}
subsequently constructed explicit lower-bound instances attaining conjectured
expressions in several two-, three-, and longer-step regimes.
Overall, exact worst-case characterizations of GD with variable
stepsize schedules remain available only in a limited number of
special cases.  
Branch-and-bound and local PEP methods have also been proposed for
computing improved stepsize schedules for prescribed horizons
\citep{das2024branch,kamri2025numerical}; note, however, that the local method does not
certify global optimality, and neither approach yields an asymptotic characterization of the
achievable convergence exponent.

Additionally, \citet{kim2024proof} derived the exact worst-case formula for
the terminal objective value of constant-step GD on smooth convex and
strongly convex functions when $\eta\in(0,2/L)$.   \citet{rotaru2026exact} established the corresponding exact formula for
the terminal gradient norm.  For possibly nonconvex $L$-smooth objectives
with a finite minimum, \citet{abbaszadehpeivasti2022exact} proved a PEP upper
bound on $\min_{0\le k\le T}\|\nabla f(x_k)\|$ for prescribed
$t_k\in(0,\sqrt{3}/L)$, and showed it to be exact when all
$t_k\in(0,1/L]$. None of these results provides a uniform worst-case analysis of the
last-iterate objective value for arbitrary nonconstant stepsize
schedules over arbitrary horizons, particularly when stepsizes larger
than $2/L$ are allowed.

\paragraph{Acceleration by (occasional) long steps.}
Reciprocal Chebyshev steps have long been known to achieve
$O(\sqrt{\kappa}\log(1/\varepsilon))$ complexity for Richardson iteration
on linear systems with positive-definite matrices, where $\kappa$ denotes the
condition number \citep{young1953richardson}.  Fractal orderings of these steps further
improve intermediate stability under bounded additive perturbations
\citep{agarwal2021fractal}. These spectral arguments, however, do not
yield worst-case guarantees for general smooth convex optimization.

More recently, several works have shown that carefully designed
stepsize schedules with occasional long steps can improve the
performance of GD beyond the
quadratic setting. PEP-certified periodic schedules improve the constant in the classical
$T^{-1}$ convergence bound \citep{grimmer2024provably}, while recursive
long-step (silver) schedules improve the dependence on the condition
number in the strongly convex setting
\citep{altschuler2025acceleration}.  The left-heavy schedule, obtained by
reversing the right-heavy schedule of Grimmer, Shu, and Wang, achieves
the same accelerated exponent for the squared norm of the terminal
gradient \citep{grimmer2025objectivegap}. Related silver acceleration
results have also been developed for projected and proximal gradient
methods \citep{bok2025proximal}. Finally, for linearly separable
logistic regression, a horizon-dependent constant stepsize of order $T$
achieves a final loss of $\widetilde O(T^{-2})$
\citep{wu2024logistic}; since the logistic loss has no finite minimizer,
this setting lies outside the class $\cF_{0,L}(\R^d)$ considered herein.

\paragraph{Anytime stepsize-based acceleration.}
\citet{kornowski2024anytime} initiated the study of anytime stepsize
schedules, asking whether a single predetermined infinite stepsize
sequence can accelerate plain GD uniformly over all stopping
times. An
anytime schedule achieving the convergence rate
$O(T^{-2\log_2\rho_{\mathrm{sil}}/(1+\log_2\rho_{\mathrm{sil}})})
=O(T^{-1.119\ldots})$ was recently constructed by
\citet{zhang2025anytime}, while \citet{tsai2026lower} proved that no
fixed positive infinite stepsize schedule can attain a worst-case
last-iterate convergence rate of $o(T^{-4/3})$. In contrast, our work
studies prescribed-horizon schedules rather than anytime schedules.

\paragraph{Other related models.}
Our lower bound concerns the unmodified last iterate. For constant
stepsizes with $L\eta\in(0,1]$, and for the dynamically increasing
Teboulle--Vaisbourd schedule, worst-case analyses show that convex
averaging cannot improve either the objective value or the gradient norm.
For the same constant-step range, the extrapolated output
$x_0+c(x_T-x_0)$, for a suitable scalar $c$, improves the worst-case
objective-value bound by a lower-order term
\citep{luner2025averaging}. Such output modifications fall outside the
method class considered in this paper.

Our model also excludes adaptive, randomized, and signed stepsize
rules. Gradient-feedback and AdaGrad-type methods use adaptive
stepsizes \citep{malitsky2020adaptive,duchi2011adaptive}, randomized
stepsize schedules follow a different model than the deterministic
predetermined schedules considered here
\citep{altschuler2024random}, and existing evidence for negative
stepsizes comes from convex--concave gradient descent--ascent rather
than smooth convex minimization \citep{shugart2025negative}.

\section{Normalized setup and a geometric construction}
\label{sec:geometric-construction}
We begin with a geometric construction that associates a
finite-dimensional hard instance with any prescribed nonnegative stepsize schedule.
The construction selects a
subset of the steps whose normalized stepsizes exceed one, preserves their original order,
and uses them to move the trajectory between orthogonal blocks.  A
Moreau envelope then realizes the resulting piecewise-constant gradient
field as the gradient of a single convex $1$-smooth function.   Maximizing the last-iterate objective gap over all selections of long steps
produces the geometric quantity that will be analyzed in
Sections~\ref{sec:chronology} and~\ref{sec:rank-analysis}.

\subsection{Normalized setting and stepsize decomposition}
\label{sec:normalized-notation}

We first analyze the normalized setting, where the smoothness
parameter and the initial distance are both equal to one, i.e.,
$L=R=1$. In Section~\ref{sec:restoration}, we will recover the general case $L,R>0$  by rescaling the objective, the iterates, and the stepsize schedule.

Specifically, let us consider a convex $1$-smooth objective function $F$ and a normalized stepsize schedule $h=(h_0,\ldots,h_{T-1})\in\R_{\ge 0}^T$. The corresponding GD updates are
\begin{equation}
        x_{t+1}=x_t-h_t\nabla F(x_t),
        \qquad 0\le t<T.
\end{equation}
The portion of each stepsize $h_t$ that exceeds the unit stepsize plays a key
role. To isolate this effect, we decompose each $h_t$ into a capped component $\min\{h_t,1\}$ and an excess component $(h_t-1)_+$, where $(u)_+\coloneqq \max\{u,0\}$. Accordingly, we define
\begin{equation}
 y_t\coloneqq (h_t-1)_+,
 \qquad
 B\coloneqq 1+\sum_{t=0}^{T-1}\min\{h_t,1\}.
 \label{eq:capped-mass}
\end{equation}
We also define the index set of long steps (i.e., those stepsizes larger than $1$) and its cardinality as follows:
\begin{equation}
 \mathcal I_+\coloneqq \{t\in\{0,\ldots,T-1\}:y_t>0\},
 \qquad r\coloneqq |\mathcal I_+|.
 \label{eq:positive-excess-indices}
\end{equation}
Thus, $r$ is the number of long steps, and clearly $r\le T$. Moreover, since each capped component of the stepsize is at most $1$, it follows immediately that $$B\le T+1.$$

\subsection{Block decomposition}

Fix any integer $m\in\{0,\ldots,r\}$.  When $m\ge1$, we choose indices
$0\le t_1<\cdots<t_m<T$ from $\mathcal I_+$ (cf.~\eqref{eq:positive-excess-indices}).  These indices identify $m$ long steps and partition the remaining iterations into the following gaps: 
\[
\begin{aligned}
 G_0& \coloneqq \{0,\ldots,t_1-1\},\\
 G_i& \coloneqq  \{t_i+1,\ldots,t_{i+1}-1\},
       &&\quad 1\le i<m,\\
 G_m& \coloneqq  \{t_m+1,\ldots,T-1\}.
\end{aligned}
\]
Note that any of these gaps may be empty. In the degenerate case $m=0$, we simply set 
$G_0 \coloneqq \{0,\ldots,T-1\}$. 

When $m>0$, for each nonterminal block $0\le i<m$, we define the {\em gap mass} $U_i$ and the associated {\em block scale} $H_i$ by
\begin{subequations}
\begin{equation}
 U_i\coloneqq 1+\sum_{t\in G_i}h_t,
 \qquad
 H_i\coloneqq U_i+y_{t_{i+1}},
 \qquad 0\leq i< m;
 \label{eq:block-masses}
\end{equation}
regarding the terminal gap, we define
\begin{equation}
 H_m:=1+2\sum_{t\in G_m}h_t.
 \label{eq:terminal-mass}
\end{equation}
In contrast, in the degenerate case $m=0$, we do not need definitions of $U_i$ but can still define
\begin{align} H_0=1+2\sum_{t=0}^{T-1}h_t,
\qquad \text{if }m=0.\end{align}
\end{subequations}

To help control the transition between consecutive blocks, we define, for each $0\le i<m$, \begin{equation} \chi_i:= \frac{y_{t_{i+1}}H_{i+1}} {U_i(H_i+H_{i+1})}. \label{eq:local-amplitude-bound} \end{equation} 
As will be shown later via a projection comparison argument, $\chi_i$ is precisely the largest admissible value of the squared amplitude ratio between blocks $i$ and $i+1$.  
Note that in the special case $m=0$, we do not need quantities $U_i$ or $\chi_i$. We are now ready to state the main result regarding realization of a schedule.

\begin{theorem}
\label{thm:geometric-realization}
Fix a nonnegative normalized stepsize schedule $h_0,\ldots,h_{T-1}$. Choose any
integer $m\in\{0,\ldots,r\}$ and, when $m\ge1$, any indices
$0\le t_1<\cdots<t_m<T$ from $\mathcal I_+$. Let the corresponding
quantities $G_i$ and $H_i$, as well as $U_i$ and $\chi_i$ when $m\ge1$,
be defined as above. Let $(\gamma_i)_{i=0}^{m-1}$ be any collection of positive amplitudes (interpreted as the empty collection when $m=0$) satisfying
\[
0<\gamma_i^2\le \chi_i,
\qquad 0\le i<m. 
\]
Then there exist a convex $1$-smooth function $F:\R^{m+1}\to\R$ and an initial
point $x_0\in\R^{m+1}$ such that $0\in\argmin F$, $\|x_0\|=1$, and gradient
descent with the schedule $h=(h_0,\ldots,h_{T-1})$ satisfies
\begin{equation}
 F(x_T)-F(0)
 =\frac1{2H_m}\prod_{i=0}^{m-1}\gamma_i^2.
 \label{eq:geometric-terminal-gap}
\end{equation}
\end{theorem}

In other words, every selection of long steps and admissible amplitude factors can be realized by a convex \(1\)-smooth instance, with the final objective value determined by the terminal block scale \(H_m\) and the product of the squared amplitude factors $(\gamma_i)_{i=0}^{m-1}$. 

\subsection{Proof of Theorem~\ref{thm:geometric-realization}}
\label{subsec:geometric-construction}

The construction is based on the classical notions of support functions and Moreau envelopes.  To be precise, for a compact convex set $K\subseteq\R^d$ obeying $0\in K$, we define its support function by
\begin{equation}
        \sigma_K(z)\coloneqq \max_{g\in K}\langle g,z\rangle,
        \label{eq:defn-sigma-K-z}
\end{equation}
as well as its associated Moreau envelope 
\begin{equation}
 \mathcal E_K(x)\coloneqq \min_{z\in\R^d}
 \left\{\sigma_K(z)+\frac12\|x-z\|^2\right\}.
 \label{eq:moreau}
\end{equation}
We also define the corresponding proximal operator by
\begin{equation}
 \operatorname{prox}_{\sigma_K}(x)
 \coloneqq \argmin_{z\in\R^d}
 \left\{\sigma_K(z)+\frac12\|x-z\|^2\right\}.
 \label{eq:proximal-point}
\end{equation}
Note that the minimizer in \eqref{eq:proximal-point} is unique because the objective is strongly convex in $z$.

We next record the standard relationship between this proximal operator, Euclidean projection onto $K$, and the gradient of the Moreau envelope.  For completeness, we include the proof of this classical result in Appendix~\ref{app:moreau-projection}.

\begin{lemma}
\label{lem:moreau-projection}
Let $K\subseteq\R^d$ be compact and convex with $0\in K$, and set
$F\coloneqq \mathcal E_K$ as defined in \eqref{eq:moreau}.  Write $\Pi_K(x)$ for
the Euclidean projection of $x$ onto $K$. Then we have
\begin{equation}
 \operatorname{prox}_{\sigma_K}(x)=x-\Pi_K(x),
 \label{eq:moreau-proximal-formula}
\end{equation}
and
\begin{equation}
 \nabla F(x)=\Pi_K(x),
 \qquad
 F(x)=\frac12\|x\|^2
      -\frac12\dist(x,K)^2.
 \label{eq:moreau-projection-formula}
\end{equation}
Consequently, $F$ is convex and its gradient is $1$-Lipschitz.
\end{lemma}

Motivated by Lemma~\ref{lem:moreau-projection}, which reveals that the gradient of $\mathcal E_K$ is simply Euclidean projection onto $K$, we shall choose $K$ as the convex hull of the origin and the prescribed block gradients, and then verify that, at each iterate visited by GD, the appropriate vertex of $K$ is the Euclidean projection. This allows a single globally defined convex $1$-smooth function to generate the desired piecewise-constant gradient pattern along the GD trajectory.

Upon fixing a block decomposition and admissible amplitude factors as in Theorem~\ref{thm:geometric-realization}, we choose an orthonormal
basis $e_0,\ldots,e_m$ of $\R^{m+1}$, and define
\[
        \lambda_0:=1, \quad \lambda_{i+1}:=\gamma_i\lambda_i,
        \quad 0\le i<m,
        \qquad
        X_i:=\lambda_i e_i,
        \quad 0\le i\le m.
\]
Next, we define the block gradients by 
\begin{equation}
 g_i:=\frac{\lambda_i}{H_i}(e_i-\gamma_ie_{i+1})
 \quad(0\le i<m),
 \qquad
 g_m:=\frac{\lambda_m}{H_m}e_m.
 \label{eq:block-gradient}
\end{equation}
We write
$$s_{\mathrm{tail}}\coloneqq \sum_{t\in G_m}h_t,
\qquad \Longrightarrow \qquad H_m=1+2s_{\mathrm{tail}},$$ 
and set
\begin{equation}
 \begin{aligned}
 K&\coloneqq\operatorname{conv}\big\{0,g_0,\ldots,g_m\big\},\\
 F(x)&\coloneqq\mathcal E_K(x)
 =\min_{z\in\R^{m+1}}
   \left\{\sigma_K(z)+\frac12\|x-z\|^2\right\},\\
 x_0&\coloneqq X_0=e_0.
 \end{aligned}
 \label{eq:geometric-objective}
\end{equation}
In view of Lemma~\ref{lem:moreau-projection}, $F$ is convex and $1$-smooth. Moreover, we have $\|x_0\|=1$ and $0\in \argmin F$ since $0\in K$.

The main ingredient of the proof is to show that, throughout block $i$, the Euclidean projection
onto $K$ is precisely the vertex $g_i$. The construction has a simple yet important local structure:
$g_i\in\operatorname{span}\{e_i,e_{i+1}\}$ for $i<m$, while
$g_m\in\operatorname{span}\{e_m\}$. Consequently, on a nonterminal block $i$,
only $g_{i-1}$ (when $i>0$), $g_i$, and $g_{i+1}$ can have a nonzero inner
product with the shifted iterate. The predecessor contribution is nonpositive,
so verifying the projection property reduces to comparing $g_i$ with its
successor $g_{i+1}$. The condition $\gamma_i^2\le\chi_i$ is exactly the
criterion ensuring that this comparison holds. Consequently, the gradient of
the Moreau envelope remains equal to $g_i$ throughout block $i$, and the GD
iterates follow the prescribed block trajectory by construction.

\begin{proof}[Proof of Theorem~\ref{thm:geometric-realization}]
For any $x,g\in\R^{m+1}$, the characterization of Euclidean projection states that
\[
 g=\Pi_K(x)
 \quad\Longleftrightarrow\quad
 g\in K
 \ \text{ and }\
 \langle x-g,v-g\rangle\le0
 \quad\text{for every }v\in K.
\]
Writing $z=x-g$, we see that $g=\Pi_K(x)$ holds if and only if $g$ maximizes
$\langle\,\cdot\,,z\rangle$ over $K$.  Since $K$ is the convex hull of
$\{0,g_0,\ldots,g_m\}$, it is therefore sufficient to verify this condition
only for the generators $0,g_0,\ldots,g_m$.

\medskip\noindent\textbf{Step 1: characterizing projection within each block.}
Our goal in this step is to verify that, throughout each block, the prescribed block gradient is the Euclidean projection onto $K$. By the discussion preceding the proof, this reduces to showing that the corresponding generator maximizes a suitable linear functional over $K$.

Specifically, for each nonterminal block $0\le i<m$ and each $s\in[0,U_i-1]$,  set
$u \coloneqq s+1\in[1,U_i]$, and define 
\begin{equation}
 \begin{aligned}
 q_i(s)&\coloneqq X_i-sg_i,\\
 z_{i,u}&\coloneqq q_i(s)-g_i
 =\frac{\lambda_i}{H_i}
   \{(H_i-u)e_i+u\gamma_ie_{i+1}\}.
 \end{aligned}
 \label{eq:block-shifted-point}
\end{equation}
The cumulative stepsize before each update in block $i$ is a value of $s$ in $[0,U_i-1]$, since
$\sum_{t\in G_i}h_t=U_i-1$.

For the terminal block, we similarly define, for every $s\in[0,s_{\mathrm{tail}}]$, 
\[
        q_m(s):=X_m-sg_m,
        \qquad
        z_m(s):=q_m(s)-g_m
        =\frac{\lambda_m(H_m-s-1)}{H_m}e_m.
\]
Note that every $z_m(s)$ is a nonnegative multiple of $e_m$, since
\[
        H_m-s-1=2s_{\mathrm{tail}}-s
        \ge s_{\mathrm{tail}}\ge0.
\]

We first analyze a nonterminal block. 
By \eqref{eq:block-gradient} and \eqref{eq:block-shifted-point},
$z_{i,u}\in\operatorname{span}\{e_i,e_{i+1}\}$, while
$g_j\in\operatorname{span}\{e_j,e_{j+1}\}$ for $j<m$ and
$g_m\in\operatorname{span}\{e_m\}$.
Since the basis vectors are orthonormal, we have
\[
        \langle g_j,z_{i,u}\rangle=0
        \qquad
        \bigl(j\notin\{i-1,i,i+1\},\ 0\le j\le m\bigr).
\]
Thus, only the predecessor $g_{i-1}$ (when $i>0$), $g_i$, and the successor $g_{i+1}$
can have nonzero inner product with $z_{i,u}$.  For $i>0$, the predecessor
satisfies
\[
 \langle g_{i-1},z_{i,u}\rangle
 =-\frac{\lambda_i^2(H_i-u)}{H_{i-1}H_i}\le0;
\]
and when $i=0$, there is no predecessor.  Consequently, we only need to
compare $g_i$ with its successor and with the zero inner products.  The
two remaining inner products are
\begin{align*}
 \langle g_i,z_{i,u}\rangle
 &=\frac{\lambda_i^2}{H_i^2}
   \left\{H_i-u(1+\gamma_i^2)\right\},\\
 \langle g_{i+1},z_{i,u}\rangle
 &=\frac{\lambda_i^2\gamma_i^2u}{H_iH_{i+1}}.
\end{align*}
The second formula remains valid for $i=m-1$, where $g_{i+1}=g_m$.
The first quantity decreases with $u$, whereas the second increases with $u$, so it
suffices to compare them at $u=U_i$.  At that endpoint, we have
\begin{align*}
 \langle g_i,z_{i,U_i}\rangle
 &=\frac{\lambda_i^2
       \{y_{t_{i+1}}-U_i\gamma_i^2\}}{H_i^2},\\
 \langle g_{i+1},z_{i,U_i}\rangle
 &=\frac{\lambda_i^2\gamma_i^2U_i}{H_iH_{i+1}}.
\end{align*}
The successor value is positive because
$\lambda_i,\gamma_i,U_i,H_i,H_{i+1}>0$.  Moreover,
\begin{align*}
 \langle g_i,z_{i,U_i}\rangle
 \ge \langle g_{i+1},z_{i,U_i}\rangle
 &\quad \iff \quad 
 \frac{y_{t_{i+1}}-U_i\gamma_i^2}{H_i^2}
 \ge\frac{\gamma_i^2U_i}{H_iH_{i+1}} \quad \iff \quad 
 \gamma_i^2\le
 \frac{y_{t_{i+1}}H_{i+1}}
 {U_i(H_i+H_{i+1})}
 =\chi_i,
\end{align*}
which holds by assumption.  Therefore, for every $1\le u\le U_i$, one has
\[
 \langle g_i,z_{i,u}\rangle
 \ge\langle g_i,z_{i,U_i}\rangle
 \ge\langle g_{i+1},z_{i,U_i}\rangle
 \ge\langle g_{i+1},z_{i,u}\rangle.
\]
In particular, the inner product with $g_i$ is positive and, therefore,  dominates the
generators with zero or negative inner product.  Consequently, $g_i$ maximizes
$\langle g,z_{i,u}\rangle$ over $g\in K$.

The terminal block is even simpler. Observe that $\langle g_m,z_m(s)\rangle\ge0$.  When $m>0$,
$\langle g_{m-1},z_m(s)\rangle\le0$, and all other generators are
orthogonal to $z_m(s)$.  Hence $g_m$ maximizes the support functional at
$z_m(s)$.  If $s_{\mathrm{tail}}=0$, then $z_m(s)=0$, so every generator
attains the same value, and in particular $g_m$ is still a maximizer.

We have therefore established the projection property throughout every block. Indeed, for every  $v\in K$,
\[
 \langle z_{i,u},v-g_i\rangle\le0,
 \qquad
 \langle z_m(s),v-g_m\rangle\le0.
\]
Since $q_i(s)=z_{i,u}+g_i$ and $q_m(s)=z_m(s)+g_m$, these are exactly the variational inequalities characterizing Euclidean
projection. It thus follows that
\[
 \Pi_K(q_i(s))=g_i,
 \qquad
 \Pi_K(q_m(s))=g_m.
\]
Applying Lemma~\ref{lem:moreau-projection} then yields
\begin{equation}
 \begin{aligned}
 \nabla F(q_i(s))&=g_i,
 &\operatorname{prox}_{\sigma_K}(q_i(s))&=z_{i,u},
 &&0\le i<m,\quad 0\le s\le U_i-1,\\
 \nabla F(q_m(s))&=g_m,
 &\operatorname{prox}_{\sigma_K}(q_m(s))&=z_m(s),
 &&0\le s\le s_{\mathrm{tail}}.
 \end{aligned}
 \label{eq:block-gradient-identity}
\end{equation}

\medskip\noindent\textbf{Step 2: characterizing the GD trajectory.}
Set
\[
 b_0:=0,
 \qquad b_i:=t_i+1\quad(1\le i\le m),
\]
so that $b_i$ is the first time index of block $i$.  For $0\le i<m$ and
$t\in G_i\cup\{t_{i+1}\}$,  define
\[
        s_{i,t}:=\sum_{\tau=b_i}^{t-1}h_\tau.
\]
Having established the gradient identities in Step 1, we now verify that the GD iterates follow the prescribed trajectory. Specifically, we claim that  throughout block \(i\),
\begin{equation}
        x_t=q_i(s_{i,t})=X_i-s_{i,t}g_i.
 \label{eq:block-time-induction}
\end{equation}

We prove the claim by induction over the blocks and, within each block, by induction over the iterations. For the first block, one has $x_{b_0}=x_0=X_0=q_0(0)$.  We now fix $i<m$ and
suppose $x_{b_i}=X_i$.  For any index
$t\in G_i\cup\{t_{i+1}\}$, we have
$0\le s_{i,t}\le U_i-1$.  Hence, whenever
\eqref{eq:block-time-induction} holds,
\eqref{eq:block-gradient-identity} gives $\nabla F(x_t)=g_i$.  If the
next iterate remains in the same block, then it holds that
\[
 x_{t+1}=x_t-h_t\nabla F(x_t)
 =X_i-(s_{i,t}+h_t)g_i
 =X_i-s_{i,t+1}g_i, 
\]
which proves the within-block induction. 
At the selected transition index $t=t_{i+1}$,
\[
 s_{i,t_{i+1}}=\sum_{t\in G_i}h_t=U_i-1,
\]
and the cumulative stepsize after the transition update is
\[
 s_{i,t_{i+1}}+h_{t_{i+1}}
 =(U_i-1)+(1+y_{t_{i+1}})=H_i.
\]
Thus the long step moves the iterate exactly to the starting point of the next block:
\[
 x_{b_{i+1}}
 =x_{t_{i+1}+1}
 =X_i-H_ig_i
 =\lambda_i\gamma_i e_{i+1}
 =X_{i+1}.
\]
This establishes the induction hypothesis for the next block and hence
verifies the entire nonterminal trajectory. 

For the terminal block, we define
\[
        s_{m,t}:=\sum_{\tau=b_m}^{t-1}h_\tau,
        \qquad b_m\le t\le T.
\]
We have $x_{b_m}=X_m$: this is the initial condition when $m=0$, and
follows from the preceding block induction when $m>0$. Applying
\eqref{eq:block-gradient-identity} and the same within-block induction gives
\[
 x_t=q_m(s_{m,t})=X_m-s_{m,t}g_m
 \quad (b_m\le t\le T).
\]
In particular,
\begin{equation}
        x_T=q_m(s_{\mathrm{tail}})
        =X_m-s_{\mathrm{tail}}g_m.
 \label{eq:terminal-iterate}
\end{equation}
This completes the characterization of the GD trajectory. Thus every update, including each zero-stepsize update, is a GD step for the same objective function $F$; the construction also covers empty gaps.

\medskip\noindent\textbf{Step 3: calculating the objective value at the last iterate.} 
It remains to evaluate the objective value at the last iterate. 
By virtue of property~\eqref{eq:terminal-iterate} and the terminal part of
\eqref{eq:block-gradient-identity},
\[
 \Pi_K(x_T)=g_m,
 \qquad
 z_T:=\operatorname{prox}_{\sigma_K}(x_T)
 =x_T-g_m=z_m(s_{\mathrm{tail}})
 =\frac{\lambda_m s_{\mathrm{tail}}}{H_m}e_m.
\]
Thus $\dist(x_T,K)=\|z_T\|$.  Using
\eqref{eq:moreau-projection-formula} and $F(0)=0$, we obtain
\[
\begin{split}
 F(x_T)-F(0)
 &=\frac12\|z_T+g_m\|^2-\frac12\|z_T\|^2=\langle g_m,z_T\rangle+\frac12\|g_m\|^2=\frac{\lambda_m^2}{H_m^2}
   \left(s_{\mathrm{tail}}+\frac12\right)
 =\frac{\lambda_m^2}{2H_m},
\end{split}
\]
where the last equality uses $H_m=1+2s_{\mathrm{tail}}$.  Finally,
since $\lambda_m^2=\prod_{i=0}^{m-1}\gamma_i^2$, we readily see that
\[
 F(x_T)-F(0)
 =\frac1{2H_m}\prod_{i=0}^{m-1}\gamma_i^2,
\]
which is precisely \eqref{eq:geometric-terminal-gap}.  The same construction covers $m=0$
and uses exactly $m+1\le T+1$ dimensions.
\end{proof}

\subsection{The resulting key functional}
\label{subsec:lower-bound-functional}

Theorem~\ref{thm:geometric-realization} shows that, for a fixed block decomposition, the final objective value is maximized by taking each amplitude factor as large as allowed, namely,  $\gamma_i^2=\chi_i$.  Optimizing further over all admissible choices of transition indices therefore leads to the functional 
\begin{equation}
 \cC_T(h)
 :=\max_{0\le m\le r}\;
   \max_{\substack{0\le t_1<\cdots<t_m<T\\
                    y_{t_j}>0\ (1\le j\le m)}}
 \frac1{H_m}\prod_{i=0}^{m-1}\chi_i.
 \label{eq:lower-bound-functional}
\end{equation}
Here, the terminal block scale $H_m$ and the local bounds $\chi_i$ are determined by
the selected indices through \eqref{eq:terminal-mass} and
\eqref{eq:local-amplitude-bound}, respectively. The purpose of $\cC_T(h)$ is
to separate the geometric realization from the remaining schedule analysis.
For a fixed schedule $h$, the quantity $\cC_T(h)/2$ is the largest terminal
objective gap attained within the family of instances constructed in
Theorem~\ref{thm:geometric-realization}. The following corollary makes this
attainment precise and thereby reduces the remaining proof to lower-bounding
$\cC_T(h)$.

\begin{corollary}
\label{cor:functional-attainment}
For every nonnegative normalized stepsize schedule $h$ of length $T$, there exist
an integer $1\le d\le T+1$, a function $F\in\cF_{0,1}(\R^d)$ with
$0\in\argmin F$, and an initial point $x_0\in\R^d$ with $\|x_0\|=1$
such that GD with schedule $h$ satisfies
\begin{equation}
 2\{F(x_T)-F(0)\}=\cC_T(h).
 \label{eq:functional-attainment}
\end{equation}
\end{corollary}

\begin{proof}
Since the maximization in \eqref{eq:lower-bound-functional} is over finitely many choices of transition indices, it is attained.  Choose indices $t_1<\cdots<t_m$ and, for each nonterminal block, set
$\gamma_i \coloneqq \sqrt{\chi_i}$.  Each
$\chi_i$ is positive by \eqref{eq:local-amplitude-bound}, and hence these
factors satisfy the assumptions of
Theorem~\ref{thm:geometric-realization}. The theorem therefore yields a convex $1$-smooth instance in dimension $d=m+1\le T+1$ satisfying
\[
 2\{F(x_T)-F(0)\}
 =\frac1{H_m}\prod_{i=0}^{m-1}\chi_i
 =\cC_T(h).
\]
If the maximum is attained at $m=0$, the same theorem applies with the empty product interpreted as one. Thus the claimed result holds in all cases.
\end{proof}
The empty selection $m=0$ is always admissible and contributes
\[
    \left(1+2\sum_{t=0}^{T-1}h_t\right)^{-1}.
\]
When all steps are short, this already gives an objective gap of order $T^{-1}$. Large
excesses can make this baseline much smaller, however, so the remaining
analysis must exploit nonempty selections.

The main difficulty is that each local factor $\chi_i$ couples two neighboring
blocks, so the value of a selected chain depends on the temporal ordering of
the long steps. Section~\ref{sec:chronology} removes this dependence through
two matchings, while Section~\ref{sec:rank-analysis} chooses an appropriate
rank cutoff and controls the residual schedule mass.

\section{Order-independent bounds via matchings}
\label{sec:chronology}

The functional $\cC_T(h)$ in \eqref{eq:lower-bound-functional} depends on both
the magnitudes and the temporal order of the selected long steps. To separate
the magnitudes of the long-step excesses from their temporal arrangement, we
first rank the positive excesses by size. If $r\ge1$, break ties arbitrarily
and label the distinct indices in $\mathcal I_+$ as
$\tau_1,\ldots,\tau_r$ so that, with $a_s:=y_{\tau_s}$,
\[
    a_1\ge a_2\ge\cdots\ge a_r>0.
\]
When $r=0$, the ranked list is empty. For $0\le q\le r$, define
\begin{equation}
    D_q:=B+\sum_{s=q+1}^{r}a_s.
 \label{eq:rank-tail}
\end{equation}
We call $D_q$ the residual schedule mass at rank $q$. It consists of the
capped base mass $B$ together with all positive excesses except the $q$
largest ones. In particular,
\[
    D_r=B,
    \qquad
    D_0=1+\sum_{t=0}^{T-1}h_t,
\]
because $h_t=\min\{h_t,1\}+y_t$.

The cutoff $q$ separates the $q$ excesses used to construct a candidate chain
from the remaining schedule mass. Once the selected excesses are restored to
chronological order, $D_q$ decomposes exactly into the intervening gap masses
and the terminal mass, as shown in \eqref{eq:mass-decomposition}. It therefore
serves as the common mass budget in the matching argument below and, later, as
the state variable in the cutoff analysis of
Section~\ref{sec:rank-analysis}.

For $2\le q\le r$, we restore the $q$ largest excesses to chronological order
and associate the reciprocals of the corresponding local factors with the
edges of a path, including a final edge for the terminal factor. The odd and
even edges form two matchings. Bounding their products by the corresponding
optimal matching values removes the temporal order, and a total reciprocal
mass estimate then gives an order-independent lower bound on $\cC_T(h)$.
Corollary~\ref{cor:functional-attainment} connects this functional bound back
to a valid normalized GD instance.

\subsection{The top-ranked excesses in temporal order}

Fix $1\le q\le r$, and consider the time labels
$\tau_1,\ldots,\tau_q$ of the $q$ largest excesses.  We define $\pi$ as the
unique permutation of $\{1,\ldots,q\}$ for which
\[
 \tau_{\pi(1)}<\cdots<\tau_{\pi(q)},
\]
which is well defined because the time labels are distinct.  We then set
\begin{align}
c_i \coloneqq a_{\pi(i)}=y_{\tau_{\pi(i)}} \qquad \text{for }1\le i\le q;
\end{align}
these are the same
excesses, now indexed by time rather than by rank.  We use the gap quantities
associated with 
$$t_i \coloneqq \tau_{\pi(i)}, \qquad 1\le i\le q,$$ and denote the schedule mass following
its last selected time by
\[
        V \coloneqq 1+\sum_{t=\tau_{\pi(q)}+1}^{T-1}h_t.
\]
The quantity $D_q$ (cf.~\eqref{eq:rank-tail}) then decomposes as
\begin{equation}
        \sum_{i=0}^{q-1}U_i+V=D_q,
        \qquad
        H_q=2V-1.
 \label{eq:mass-decomposition}
\end{equation}
The first identity holds because each selected long step contributes
its unit part to $B$, each unselected excess remains in $D_q$, and every other schedule term belongs either to one of the gaps or to the terminal portion $V$.  Since $V\ge1$, \eqref{eq:mass-decomposition} also immediately gives
\begin{equation}
        \sum_{i=0}^{q-1}U_i\le D_q-1,
        \qquad
        U_{q-1}+H_q\le2D_q-1,
\end{equation}
where the second inequality follows because 
$U_{q-1}+H_q\le\sum_iU_i+2V-1=D_q+V-1\le2D_q-1$.
For these selected indices, $H_i=U_i+c_{i+1}$ for $0\le i<q$, and the local factors are
those in \eqref{eq:local-amplitude-bound}.  We expand their reciprocals as
\begin{equation}
\begin{aligned}
 \chi_{i-1}^{-1}
 &=\frac{U_{i-1}}{c_i}
   +\frac{U_{i-1}}{H_i}
   +\frac{U_{i-1}^2}{c_iH_i},
 &&1\le i<q,\\
 H_q\chi_{q-1}^{-1}
 &=U_{q-1}\left(1+\frac{U_{q-1}+H_q}{c_q}\right).
\end{aligned}
 \label{eq:local-factor-identities}
\end{equation}

\subsection{A path bound and its two matchings}

Each nonterminal factor $\chi_{i-1}^{-1}$, $1\le i<q$, couples two consecutive excesses
$c_i,c_{i+1}$ in temporal order.  The terminal factor involves only
$c_q$; after normalization, we represent it by adjoining one auxiliary terminal
vertex.  This produces a path on $q+1$ vertices.  Alternating edges along
the path share no endpoints, which is why two matchings, one for each
parity, give a bound that no longer depends on the particular temporal ordering of the selected excesses.

For $u,v>0$,  define
\begin{equation}
        \psi(u,v)=\frac{u+v+uv}{2}.
 \label{eq:psi-kernel}
\end{equation}
For a finite label set $\mathcal V$, let
$W=(w_\alpha)_{\alpha\in\mathcal V}$ be a family of positive vertex weights. For an integer $1\le k\le\lfloor |\mathcal V|/2\rfloor$, define $P_\psi(W,k)$ as the maximum product of $\psi$ over all
$k$-edge matchings on the labeled elements of $W$:
\[
P_\psi(W,k)
:=\max_{\substack{
\alpha_1,\beta_1,\ldots,\alpha_k,\beta_k\in\mathcal V\\
\text{all distinct}}}
\prod_{j=1}^k\psi(w_{\alpha_j},w_{\beta_j}),
\qquad
P_\psi(W,0):=1.
\]

For a rank cutoff $2\le q\le r$, we define
\begin{equation}
 w_s^{(q)}=\frac{2D_q}{qa_s}\quad(1\le s\le q),
 \qquad
 w_\dagger^{(q)}=\frac1{q-1},
 \qquad
 W_q=\Big\{w_1^{(q)},\ldots,w_q^{(q)},w_\dagger^{(q)}\Big\}.
 \label{eq:matching-vertices}
\end{equation}
The first $q$ weights are normalized reciprocals of the selected excesses,
while $w_\dagger^{(q)}$ is the auxiliary terminal weight.  The multiset
$W_q$ is labeled, so equal numerical weights remain distinct.  We then set
\begin{equation}
 k_+=\left\lceil\frac q2\right\rceil,
 \qquad
 k_-=\left\lfloor\frac q2\right\rfloor,
 \qquad
 \cM_q=P_\psi(W_q,k_+)P_\psi(W_q,k_-).
 \label{eq:matching-product}
\end{equation}

In temporal order, these vertex weights are
\[
 v_i:=w_{\pi(i)}^{(q)}=\frac{2D_q}{qc_i}\quad(1\le i\le q),
 \qquad
 v_{q+1}:=w_\dagger^{(q)}.
\]
Thus, the path corresponding to the chosen excesses is
\[
 \underbrace{v_1\to\cdots\to v_q}_{
   \text{normalized reciprocals of selected excesses}}
 \longrightarrow
 \underbrace{v_{q+1}=w_\dagger^{(q)}}_{
   \text{auxiliary terminal vertex}}.
\]
We now have the following order-independent lower bound. 

\begin{proposition}
\label{prop:endpoint-matching}
For every $2\le q\le r$, one has
\begin{equation}
        \cC_T(h)
        \ge\frac{q}{2D_q(q-1)\cM_q}.
 \label{eq:order-independent-matching-bound}
\end{equation}
Importantly, this bound is independent of the temporal order of the top $q$
excesses.
\end{proposition}

\subsection{Proof of Proposition~\ref{prop:endpoint-matching}}

We first record the following product inequality, whose proof is deferred to Appendix~\ref{app:budget-equalization}.
\begin{lemma}
\label{lem:budget-equalization}
Let $q\ge1$ and $D>0$, and put $\bar u:=D/q$.  If
$u_1,\ldots,u_q>0$ satisfy $\sum_{i=1}^q u_i\le D$, and
$A_i,B_i>0$ for $1\le i<q$, then
\begin{equation}
 \left(\prod_{i=1}^q u_i\right)
 \prod_{i=1}^{q-1}(A_i+B_iu_i)
 \le
 \bar u^q\prod_{i=1}^{q-1}(A_i+2\bar u B_i).
 \label{eq:budget-product}
\end{equation}
\end{lemma}

\begin{proof}[Proof of Proposition~\ref{prop:endpoint-matching}]
We write $\xi_i:=c_i^{-1}$ for the reciprocal of the $i$th excess in
temporal order.  The identities in
\eqref{eq:mass-decomposition} give
\[
        \sum_{i=0}^{q-1}U_i\le D_q-1<D_q,
        \qquad U_{q-1}+H_q\le2D_q-1<2D_q.
\]
Moreover, $U_i\ge1$ for $0\le i<q$ and $H_q\ge1$ because the schedule is
nonnegative.  Since every $c_i$ is positive, so are all the quantities
$\chi_i$ in \eqref{eq:local-amplitude-bound}.

For $1\le i<q$, the first identity in
\eqref{eq:local-factor-identities}, together with
$H_i=U_i+c_{i+1}\ge c_{i+1}$, gives
\begin{equation}
 \chi_{i-1}^{-1}
 \le U_{i-1}(\xi_i+\xi_{i+1})+U_{i-1}^2\xi_i\xi_{i+1}.
 \label{eq:nonterminal-factor-bound}
\end{equation}
For the terminal factor, the second identity in
\eqref{eq:local-factor-identities} and
$U_{q-1}+H_q<2D_q$ give
\begin{equation}
 H_q\chi_{q-1}^{-1}
 \le U_{q-1}(1+2D_q\xi_q).
 \label{eq:terminal-factor-bound}
\end{equation}

For $1\le i<q$, we set
\[
        S_i:=\xi_i+\xi_{i+1},
        \qquad
        \Xi_i:=\xi_i\xi_{i+1}.
\]
Although the individual gap masses are not controlled separately, their sum
is less than $D_q$.  We multiply \eqref{eq:nonterminal-factor-bound} over
$1\le i<q$ and then use \eqref{eq:terminal-factor-bound} to obtain
\[
 H_q\prod_{j=0}^{q-1}\chi_j^{-1}
 \le(1+2D_q\xi_q)\left(\prod_{j=0}^{q-1}U_j\right)
       \prod_{i=1}^{q-1}(S_i+\Xi_iU_{i-1}).
\]
We can therefore apply Lemma~\ref{lem:budget-equalization} with
$D=D_q$, $u_i=U_{i-1}$, $A_i=S_i$, and $B_i=\Xi_i$, where its hypotheses follow
from $U_i\ge1$, $\sum_{i=0}^{q-1}U_i<D_q$, and $S_i,\Xi_i>0$.  This yields
\begin{equation}
\begin{split}
 H_q\prod_{i=0}^{q-1}\chi_i^{-1}
 &\le
 (1+2D_q\xi_q)
 \left(\frac{D_q}{q}\right)^q
 \prod_{i=1}^{q-1}
 \left(S_i+\frac{2D_q}{q}\Xi_i\right)\\
 &=2D_q\left(\frac{D_q}{q}\right)^q
 \left(\xi_q+\frac1{2D_q}\right)
 \prod_{i=1}^{q-1}
 \left(\xi_i+\xi_{i+1}
       +\frac{2D_q}{q}\xi_i\xi_{i+1}\right).
\end{split}
 \label{eq:factorized-local-bound}
\end{equation}

We now use the normalization $v_i=2D_q\xi_i/q$ to express every factor in
terms of the same function $\psi$.  More precisely,
\begin{equation}
 \frac{D_q}{q}\left(S_i+\frac{2D_q}{q}\Xi_i\right)
 =\psi(v_i,v_{i+1})\quad(1\le i<q),
 \qquad
 2D_q\frac{D_q}{q}\left(\xi_q+\frac1{2D_q}\right)
 =\frac{2D_q(q-1)}{q}\psi(v_q,v_{q+1}).
 \label{eq:terminal-edge-identity}
\end{equation}
Both sides of the second identity equal
$(D_q/q)(1+2D_q\xi_q)$, so adjoining the auxiliary vertex introduces no
further inequality: it merely puts the terminal factor in the same
two-variable form as the internal factors.  The resulting path has labeled
vertex multiset $W_q$, with the selected excesses in their actual temporal
order.  We denote its edge product by
\[
        \mathcal P_{\rm path}
        :=\prod_{i=1}^q\psi(v_i,v_{i+1}).
\]
Taken together, \eqref{eq:factorized-local-bound} and
\eqref{eq:terminal-edge-identity} give
\begin{equation}
 H_q\prod_{i=0}^{q-1}\chi_i^{-1}
 \le\frac{2D_q(q-1)}{q}\mathcal P_{\rm path}.
 \label{eq:augmented-path}
\end{equation}

We next remove the dependence on that order by splitting the path edges
according to parity:
\[
 E_+ \coloneqq \big\{\{v_i,v_{i+1}\}:1\le i\le q,\ i\ \text{odd}\big\},
 \qquad
 E_- \coloneqq  \big\{\{v_i,v_{i+1}\}:1\le i\le q,\ i\ \text{even}\big\}.
\]
No two edges of the same parity share a vertex.  Hence $E_+$ and $E_-$ are
matchings on $W_q$ with $k_+$ and $k_-$ edges, respectively, and the
definition of $P_\psi$ yields
\[
 \mathcal P_{\rm path}
 =\left(\prod_{\{u,v\}\in E_+}\psi(u,v)\right)
  \left(\prod_{\{u,v\}\in E_-}\psi(u,v)\right)
 \le P_\psi(W_q,k_+)P_\psi(W_q,k_-)=\cM_q.
\]
We combine this inequality with \eqref{eq:augmented-path} and take positive
reciprocals to obtain
\[
\begin{split}
 \frac1{H_q}\prod_{i=0}^{q-1}\chi_i
 &=\left(H_q\prod_{i=0}^{q-1}\chi_i^{-1}\right)^{-1}
 \ge\frac{q}{2D_q(q-1)\mathcal P_{\rm path}}
 \ge\frac{q}{2D_q(q-1)\cM_q}.
\end{split}
\]
The left-hand side is the value in \eqref{eq:lower-bound-functional} corresponding to the indices of the $q$ largest excesses. Since $\cC_T(h)$ is the maximum over all admissible choices of selected indices, \eqref{eq:order-independent-matching-bound} follows.
\end{proof}

\subsection{Bounding the matchings by total weight}
\label{subsec:matching-total-weight}

The quantity $\cM_q$ no longer depends on the temporal order of the selected excesses, but it still
depends on each of the first $q$ excesses separately.  We now bound it using only their scaled reciprocal sum.  For $1\le q\le r$,  define
\begin{equation}
 \zeta_q \coloneqq \frac{D_q}{q^2}\sum_{s=1}^q\frac1{a_s}.
 \label{eq:reciprocal-mass}
\end{equation}
Thus, $\zeta_q$ is $D_q/q$ times the mean reciprocal of the first $q$
excesses.  For $2\le q\le r$, we also write
\begin{equation}
        \mu_q:=\cM_q^{1/q}.
 \label{eq:normalized-matching-factor}
\end{equation}
Since the two matching products contain
$k_++k_-=q$ edges in total, $\mu_q$ is their geometric mean per edge.  The
sum of the $q$ excess-dependent vertex weights is
\[
        \sum_{s=1}^q w_s^{(q)}=2q\zeta_q,
\]
and the auxiliary terminal vertex adds $(q-1)^{-1}$.  Thus it remains to bound a matching product in terms of the total weight of its vertices.

\begin{lemma}
\label{lem:total-mass-matching}
Let $W$ be a finite labeled multiset of positive weights, set
$\Sigma:=\sum_{w\in W}w$, and let
$1\le k\le\lfloor |W|/2\rfloor$.  Then
\begin{equation}
 P_\psi(W,k)^{1/k}
 \le \frac{\Sigma}{2k}+\frac{\Sigma^2}{8k^2}.
 \label{eq:total-mass-matching}
\end{equation}
Consequently, for every $2\le q\le r$,
\begin{equation}
 \mu_q
 \le
 \frac{2q\zeta_q+(q-1)^{-1}}{q-1}
 +\frac{\{2q\zeta_q+(q-1)^{-1}\}^2}{2(q-1)^2}.
 \label{eq:scalar-matching-bound}
\end{equation}
\end{lemma}

\begin{proof}
Choose a matching attaining $P_\psi(W,k)$, and list the weights of its
selected vertices as
$0<x_1\le\cdots\le x_{2k}$.  We first note that these vertices may be
paired from the outside inward.  To see this, we take
$0<a\le b\le c\le d$ and set $Y_s:=1+s$.  Since
$2\psi(s,t)=Y_sY_t-1$, direct subtraction gives
\begin{align*}
 4\{\psi(a,d)\psi(b,c)-\psi(a,c)\psi(b,d)\}
 &=(Y_b-Y_a)(Y_d-Y_c)\ge0,\\
 4\{\psi(a,d)\psi(b,c)-\psi(a,b)\psi(c,d)\}
 &=(Y_c-Y_a)(Y_d-Y_b)\ge0.
\end{align*}
We apply these inequalities with $a=x_1$ and $d=x_{2k}$.  If the corresponding
labeled vertices are not paired with one another, we write $b\le c$ for the
weights of their current partners.  Those two edges contribute either
$\psi(a,b)\psi(c,d)$ or $\psi(a,c)\psi(b,d)$, so the two inequalities show
that replacing them by $\{a,d\}$ and $\{b,c\}$ does not decrease the
product.  The remaining pairs must still be optimal on the remaining
selected vertices.  Repeating the same argument recursively, we obtain a maximizing matching that pairs $x_i$ with $x_{2k+1-i}$ for every $i$.

Set $\Sigma_E \coloneqq \sum_{i=1}^{2k}x_i$.  The sequence $(x_i)_{i=1}^k$ is
nondecreasing, whereas $(x_{2k+1-i})_{i=1}^k$ is nonincreasing.  Hence
\[
 \begin{split}
 &k\sum_{i=1}^k x_ix_{2k+1-i}
 -\left(\sum_{i=1}^k x_i\right)
  \left(\sum_{i=1}^k x_{2k+1-i}\right)=\sum_{1\le i<j\le k}
   (x_i-x_j)(x_{2k+1-i}-x_{2k+1-j})
 \le0.
 \end{split}
\]
Therefore, the reverse Chebyshev inequality and the elementary inequality 
$ab\le(a+b)^2/4$ give
\[
 \sum_{i=1}^k x_ix_{2k+1-i}
 \le\frac1k
 \left(\sum_{i=1}^k x_i\right)
 \left(\sum_{i=k+1}^{2k}x_i\right)
 \le\frac{\Sigma_E^2}{4k}.
\]
The arithmetic--geometric mean inequality for the $k$ edge weights now
yields
\[
\begin{split}
 P_\psi(W,k)^{1/k}
 &\le\frac1k\sum_{i=1}^k\psi(x_i,x_{2k+1-i})=\frac{\Sigma_E}{2k}
   +\frac1{2k}\sum_{i=1}^kx_ix_{2k+1-i}\le\frac{\Sigma_E}{2k}+\frac{\Sigma_E^2}{8k^2}
 \le\frac{\Sigma}{2k}+\frac{\Sigma^2}{8k^2}.
\end{split}
\]
The last inequality uses $\Sigma_E\le\Sigma$ and the fact that
$x\mapsto x/(2k)+x^2/(8k^2)$ is increasing on $[0,\infty)$.
This proves \eqref{eq:total-mass-matching}.

For the particular vertex multiset $W_q$, we write its total weight as
\[
 \Sigma_q
 :=\sum_{w\in W_q}w
 =2q\zeta_q+\frac1{q-1}.
\]
For fixed $\Sigma_q>0$, the right-hand side of
\eqref{eq:total-mass-matching} decreases with $k$.  Since
$k_+,k_-\ge(q-1)/2$, we can apply the preceding bound to each matching to
obtain
\[
 P_\psi(W_q,k_\pm)^{1/k_\pm}
 \le \frac{\Sigma_q}{q-1}+\frac{\Sigma_q^2}{2(q-1)^2}.
\]
Finally, using $k_++k_-=q$, we reach
\[
\begin{split}
 \mu_q
 &=\left(P_\psi(W_q,k_+)^{1/k_+}\right)^{k_+/q}
   \left(P_\psi(W_q,k_-)^{1/k_-}\right)^{k_-/q}\le \frac{\Sigma_q}{q-1}+\frac{\Sigma_q^2}{2(q-1)^2},
\end{split}
\]
which is precisely \eqref{eq:scalar-matching-bound}.
\end{proof}

\section{A cutoff argument and completion of the proof} \label{sec:rank-analysis}

For each rank cutoff $2\le q\le r$, the preceding section gives an order-independent matching bound expressed in terms of $\mu_q$, the geometric mean per edge of the matching product. We compare this quantity with a fixed threshold $\rho<1$. For all sufficiently large ranks $q\ge Q$, this comparison yields two alternatives. If $\mu_q<\rho$, then the matching bound already produces a large contribution to the objective gap. If $\mu_q\ge\rho$, then the smallest excess among the first $q$ ranked excesses must be small relative to the remaining schedule mass. When this second alternative persists over many ranks, a Lyapunov argument controls the cumulative growth of the residual schedule mass. Once the argument has reduced the problem to only a fixed number of remaining ranks, a direct prefix estimate completes the lower bound. The admissible parameter choices lead exactly to the threshold $p_\star=\sqrt{2+\sqrt3}$. Finally, we restore the original scaling and complete the proof of Theorem~\ref{thm:main}.

\subsection{Choice of parameters and the two cutoff alternatives}
\label{subsec:cutoff-alternatives}

Recall that \eqref{eq:order-independent-matching-bound} holds for every
$2\le q\le r$.  In this bound, $D_q$ is the schedule mass left after the
$q$ largest excesses have been selected, while
$\mu_q^{\,q}=\cM_q$ is the corresponding matching product.  We will choose
a rank at which either this product is small or the fact that it is not small
gives control of $D_q$.

Our target is the following schedule-independent lower bound.

\begin{proposition}
\label{prop:normalized-floor}
For every fixed $p\in(p_\star,2)$, there is
$c_p>0$, depending only on $p$, such that every nonnegative normalized stepsize schedule $h$
satisfies
\begin{equation}
        \cC_T(h)
        \ge\frac{c_p}{B(r+1)^{p-1}}.
 \label{eq:normalized-lower-bound}
\end{equation}
\end{proposition}

For the rest of the normalized argument, we fix $p\in(p_\star,2)$ and define
\begin{equation}
        \vartheta:=\frac1{p^2-1}.
 \label{eq:parameter-choice}
\end{equation}
Then we have
\[
\begin{aligned}
 2\vartheta+2\vartheta^2<1
 \quad\Longleftrightarrow\quad p^2>2+\sqrt3.
\end{aligned}
\]
Thus, the strict inequality $p>p_\star$ allows us to choose
$\rho\in(0,1)$ so that
\begin{equation}
        2\vartheta+2\vartheta^2<\rho<1.
 \label{eq:rho-choice}
\end{equation}
At $p=p_\star$, the left endpoint equals one, so no such $\rho$ exists.
This explains why the argument requires the strict inequality
$p>p_\star$.

We next record how much the residual schedule mass changes when we move the cutoff by one
rank.  For $1\le q\le r$, we define
\begin{equation}
        \nu_q:=\frac{qa_q}{D_q}.
 \label{eq:relative-mass-increment}
\end{equation}
Because $D_{q-1}=D_q+a_q$, the quotient $\nu_q/q$ is the relative
increase of $D_q$ when the cutoff moves from $q$ to $q-1$.  This quantity
is related to the matching quantity $\zeta_q$ from
\eqref{eq:reciprocal-mass}.  Indeed, since $a_s\ge a_q$ for $s\le q$,
\begin{equation}
 \zeta_q\nu_q
 =\frac{a_q}{q}\sum_{s=1}^q\frac1{a_s}
 \le1.
 \label{eq:zeta-density-bound}
\end{equation}

If $\zeta_q\le\vartheta$, then
\[
 \frac{2q\zeta_q+(q-1)^{-1}}{q-1}
 \le 2\vartheta+\frac{2\vartheta}{q-1}
      +\frac1{(q-1)^2}
 \le 2\vartheta+\frac{2\vartheta+1}{q-1},
\]
where the last inequality uses $q\ge2$.  Since $x\mapsto x+x^2/2$ is
increasing on $[0,\infty)$, \eqref{eq:scalar-matching-bound} gives
\[
 \mu_q\le
 2\vartheta+\frac{2\vartheta+1}{q-1}
 +\frac12\left(2\vartheta+
          \frac{2\vartheta+1}{q-1}\right)^2.
\]
As $q\to\infty$, the right-hand side decreases to
$2\vartheta+2\vartheta^2<\rho$.  We therefore choose an integer $Q\ge2$
large enough such that
\[
 Q>\vartheta^{-1},
 \qquad
 2\vartheta+\frac{2\vartheta+1}{Q-1}
 +\frac12\left(2\vartheta+
          \frac{2\vartheta+1}{Q-1}\right)^2<\rho.
\]
With this choice, uniformly for every $Q\le q\le r$, one has
\begin{equation}
        \zeta_q\le\vartheta\quad\Longrightarrow\quad \mu_q<\rho.
 \label{eq:uniform-matching-bound}
\end{equation}
Both $\rho$ and $Q$ depend only on $p$ and will remain fixed.

For each $Q\le q\le r$, we now consider the two cases
$\mu_q<\rho$ and $\mu_q\ge\rho$.  If $\mu_q<\rho$, then combining
\eqref{eq:order-independent-matching-bound} with $q/(q-1)>1$ gives
\begin{equation}
 \mu_q<\rho
 \quad\Longrightarrow\quad
 \cC_T(h)>\frac{\rho^{-q}}{2D_q}.
 \label{eq:small-matching-product}
\end{equation}
On the other hand, if $\mu_q\ge\rho$, we first apply the contrapositive of
\eqref{eq:uniform-matching-bound} and then use
\eqref{eq:zeta-density-bound} to obtain
\begin{equation}
 \mu_q\ge\rho
 \quad\Longrightarrow\quad
 \zeta_q>\vartheta,
 \qquad 0<\nu_q<\vartheta^{-1}.
 \label{eq:large-matching-product}
\end{equation}
In the proof below, we take the largest $q$ for which the first alternative
holds.  At all larger ranks the second alternative holds.  If there is no
such $q$, the second alternative holds throughout $\{Q,\ldots,r\}$.

\subsection{Growth of the residual schedule mass and the bounded-rank case}
\label{subsec:mass-growth}

The pointwise bounds in \eqref{eq:large-matching-product} show that, under
the large-product alternative, each individual rank can increase the
residual schedule mass only by a controlled relative amount.  However, these local
bounds do not by themselves control the cumulative growth of $D_q$ across
many ranks.  We therefore pass from pointwise control to an aggregate
estimate.  First, we record the exact relations between adjacent ranks; then
we introduce a Lyapunov potential tailored to the
critical exponent $p-1$.

\subsubsection{Adjacent-rank dynamics}
\label{subsubsec:adjacent-rank-dynamics}

The following lemma collects the exact relations between
successive ranks used in the Lyapunov analysis; its proof is deferred to Appendix~\ref{app:adjacent-rank-relations}.

\begin{lemma}
\label{lem:adjacent-rank-relations}
For every $1\le q\le r$,
\begin{equation}
 \frac{D_{q-1}}{D_q}=1+\frac{\nu_q}{q}.
 \label{eq:one-rank-mass-ratio}
\end{equation}
Consequently, for $0\le k\le r$,
\begin{equation}
 \frac{D_k}{B}
 =\prod_{s=k+1}^r\left(1+\frac{\nu_s}{s}\right).
 \label{eq:mass-product}
\end{equation}
For $1\le q<r$, monotonicity of the ranked excesses implies
\begin{equation}
 \frac{\nu_q}{q}
 \left(1+\frac{\nu_{q+1}}{q+1}\right)
 \ge\frac{\nu_{q+1}}{q+1},
 \label{eq:mass-increment-transition}
\end{equation}
and hence, whenever $q>\nu_q$,
\begin{equation}
        \nu_{q+1}
        \le\frac{(q+1)\nu_q}{q-\nu_q}.
 \label{eq:mass-increment-bound}
\end{equation}
The matching quantity satisfies the exact recursion
\begin{equation}
 \zeta_{q+1}
 =\frac{q^2\zeta_q}{(q+1)(q+1+\nu_{q+1})}
  +\frac1{\nu_{q+1}(q+1)}
 \qquad(1\le q<r).
 \label{eq:zeta-recursion}
\end{equation}
\end{lemma}

\subsubsection{A one-step Lyapunov estimate}
\label{subsubsec:one-step-lyapunov}

To sum the changes in $D_q$, we need to control
$\sum_q \nu_q/q$.  The affine recursion \eqref{eq:zeta-recursion}
suggests the Lyapunov potential
\begin{align}
 \mathcal L_q
 \coloneqq \frac{\nu_q(\zeta_q-\vartheta)}
         {\vartheta(\nu_q+p+1)}.
\end{align}
The coefficient of $\zeta_q-\vartheta$ is chosen so that the contraction in
\eqref{eq:zeta-recursion} is absorbed by the preceding value of the
potential.  The resulting drift estimate bounds the relative mass increase,
up to a summable quadratic error.

\begin{lemma}
\label{lem:one-step-lyapunov}
There is a constant $C_p\ge1$, depending only on $p$, with the following
property.  Let $n\ge2$ be an integer satisfying
$n-1>\vartheta^{-1}$.  Suppose
\[
 \zeta_{n-1}\ge\vartheta,
 \qquad
 0<\nu_{n-1},\nu_n\le\vartheta^{-1},
\]
and suppose
\[
 \nu_n\le
 \frac{n\nu_{n-1}}{n-1-\nu_{n-1}},
 \qquad
 \zeta_n=
 \frac{(n-1)^2\zeta_{n-1}}{n(n+\nu_n)}
 +\frac1{n\nu_n}.
\]
Then
\begin{equation}
 \mathcal L_n-\mathcal L_{n-1}
 \le\frac{p-1-\nu_n}{n}
      +\frac{C_p}{n^2}.
 \label{eq:one-step-lyapunov}
\end{equation}
\end{lemma}

The proof consists of a coefficient comparison followed by a direct
calculation of the drift; we defer it to
Appendix~\ref{app:one-step-lyapunov-proof}.  Telescoping this inequality
then leads to the required bound on $D_q$.

\begin{lemma}
\label{lem:boundary-propagation}
There is a constant $K_p\ge1$, depending only on $p$, such that, whenever
$Q\le k\le r$ and $\mu_q\ge\rho$ for every
$q\in\{k+1,\ldots,r\}$,
\begin{equation}
        D_k k^{p-1}
        \le K_pB r^{p-1}.
 \label{eq:boundary-propagation}
\end{equation}
\end{lemma}

\begin{proof}[Proof of Lemma~\ref{lem:boundary-propagation}]
We take $C_p\ge1$ to be the constant from
Lemma~\ref{lem:one-step-lyapunov} and define
\[
 K_p:=2\exp\!\left\{
   \frac1{\vartheta(p+1)}+C_p\frac{\pi^2}{6}
 \right\}.
\]
This constant depends only on $p$.

\medskip\noindent\textbf{Estimate on an interval of ranks.}
We first show that, whenever $Q\le\ell\le r$ and
$\mu_q\ge\rho$ for every $q\in\{\ell,\ldots,r\}$,
\begin{equation}
        D_\ell\ell^{p-1}
        \le \frac{K_p}{2}B r^{p-1}.
 \label{eq:core-propagation}
\end{equation}
If $\ell=r$, this follows from $D_r=B$ and $K_p/2\ge1$.  We now assume
$\ell<r$.  Applying \eqref{eq:large-matching-product} and
\eqref{eq:zeta-density-bound} at each of these ranks gives
\begin{equation}
 \zeta_q>\vartheta,
 \qquad
 0<\nu_q<\vartheta^{-1},
 \qquad
 \zeta_q\nu_q\le1
 \quad(\ell\le q\le r).
 \label{eq:cutoff-conditions}
\end{equation}
Moreover, for every $q\in\{\ell,\ldots,r-1\}$,
$q\ge Q>\vartheta^{-1}>\nu_q$, so \eqref{eq:mass-increment-bound} applies to
every transition used below.

Specifically, we fix $n\in\{\ell+1,\ldots,r\}$ and apply
\eqref{eq:mass-increment-bound} and \eqref{eq:zeta-recursion} with $q=n-1$.
They give
\[
 \nu_n\le
 \frac{n\nu_{n-1}}{n-1-\nu_{n-1}},
 \qquad
 \zeta_n=
 \frac{(n-1)^2\zeta_{n-1}}{n(n+\nu_n)}
 +\frac1{n\nu_n}.
\]
Together with \eqref{eq:cutoff-conditions}, these are precisely the hypotheses of
Lemma~\ref{lem:one-step-lyapunov}.  Hence that lemma is applicable for every
$n\in\{\ell+1,\ldots,r\}$.

In view of 
\eqref{eq:cutoff-conditions}, the potential $\mathcal L_q$ is uniformly bounded as follows: 
\begin{equation}
 0\le\mathcal L_q
 \le\frac{\nu_q\zeta_q}
          {\vartheta(\nu_q+p+1)}
 \le\frac1{\vartheta(\nu_q+p+1)}
 \le\frac1{\vartheta(p+1)}.
 \label{eq:lyapunov-bound}
\end{equation}
In particular, the endpoint contribution satisfies
\[
        \mathcal L_\ell-\mathcal L_r
        \le\mathcal L_\ell
        \le\frac1{\vartheta(p+1)}.
\]
For each $n\in\{\ell+1,\ldots,r\}$, we rearrange
\eqref{eq:one-step-lyapunov} to obtain
\[
 \frac{\nu_n}{n}
 \le\frac{p-1}{n}+\frac{C_p}{n^2}
   +\mathcal L_{n-1}-\mathcal L_n.
\]
We now sum these inequalities and telescope the potential terms to obtain
\[
\begin{split}
 \sum_{n=\ell+1}^r\frac{\nu_n}{n}
 &\le
 (p-1)\sum_{n=\ell+1}^r\frac1n
 +C_p\sum_{n=\ell+1}^r\frac1{n^2}
 +\sum_{n=\ell+1}^r
   \bigl(\mathcal L_{n-1}-\mathcal L_n\bigr)\\
 &=
 (p-1)\sum_{n=\ell+1}^r\frac1n
 +C_p\sum_{n=\ell+1}^r\frac1{n^2}
 +\mathcal L_\ell-\mathcal L_r\\
 &\le
 (p-1)\log\frac r\ell
 +\frac1{\vartheta(p+1)}
 +C_p\frac{\pi^2}{6}.
\end{split}
\]
Here, we have used \eqref{eq:lyapunov-bound},
$\sum_{n=\ell+1}^r n^{-1}\le\log(r/\ell)$, and
$\sum_{n=1}^\infty n^{-2}=\pi^2/6$.
Finally, we set $k=\ell$ in \eqref{eq:mass-product}, take logarithms, and use
$\log(1+x)\le x$ to derive
\[
\begin{split}
 \log\frac{D_\ell}{B}
 &=\sum_{n=\ell+1}^r
   \log\left(1+\frac{\nu_n}{n}\right)\le\sum_{n=\ell+1}^r\frac{\nu_n}{n}\le(p-1)\log\frac r\ell
 +\frac1{\vartheta(p+1)}
 +C_p\frac{\pi^2}{6}.
\end{split}
\]
We then exponentiate this inequality and multiply by $B\ell^{p-1}$ to reach
\[
 D_\ell\ell^{p-1}
 \le
 \exp\!\left\{
   \frac1{\vartheta(p+1)}+C_p\frac{\pi^2}{6}
 \right\}Br^{p-1}
 =\frac{K_p}{2}Br^{p-1},
\]
which proves \eqref{eq:core-propagation}.

\medskip\noindent\textbf{The remaining boundary rank.}
If $k=r$, then
\[
        D_k k^{p-1}=D_r r^{p-1}=B r^{p-1}
        \le K_p B r^{p-1}.
\]
We now suppose $k<r$.  Since $\mu_q\ge\rho$ for every
$q\in\{k+1,\ldots,r\}$, \eqref{eq:core-propagation} at $\ell=k+1$
gives
\[
        D_{k+1}(k+1)^{p-1}
        \le \frac{K_p}{2}B r^{p-1}.
\]
Moreover, $\mu_{k+1}\ge\rho$, and hence  \eqref{eq:large-matching-product} gives
$\nu_{k+1}<\vartheta^{-1}$.  Therefore, it holds that
\[
 \frac{D_k}{D_{k+1}}
 =1+\frac{\nu_{k+1}}{k+1}
 <1+\frac{\vartheta^{-1}}{k+1}
 \le1+\frac{\vartheta^{-1}}Q
 <2,
\]
where the last inequality uses $Q>\vartheta^{-1}$.  Since
$k^{p-1}\le(k+1)^{p-1}$, we can conclude that
\[
 D_k k^{p-1}
 <2D_{k+1}(k+1)^{p-1}
 \le K_pB r^{p-1},
\]
thereby establishing \eqref{eq:boundary-propagation}.
\end{proof}

\subsubsection{The bounded-rank case}

The preceding estimate controls the residual schedule mass down to any fixed rank. It remains to extract a lower bound using only a fixed number of the largest excesses. The key point is that, among the empty chain and the prefix chains formed from these excesses, at least one achieves a value comparable to the reciprocal of the remaining mass. More concretely, either the empty chain already suffices, or one of the nonempty prefixes has well-controlled local factors.

\begin{lemma} \label{lem:bounded-rank} 
For every fixed integer $Q\ge1$, there exists a constant $c_Q>0$, depending only on $Q$, such that, for every $0\le q_0\le\min\{Q,r\}$, \begin{equation} \cC_T(h)\ge\frac{c_Q}{D_{q_0}}. \label{eq:bounded-rank-estimate} \end{equation} \end{lemma} 
The lemma follows by applying Lemma~\ref{lem:bounded-rank-prefix} at rank $q_0$. The proof details are postponed to Appendix~\ref{app:bounded-rank}.

\subsection{Proof of Proposition~\ref{prop:normalized-floor}}
\label{subsec:normalized-lower-bound}

We now combine the growth estimate with the bounded-rank argument to prove Proposition~\ref{prop:normalized-floor}.

Let $K_p\ge 1$ be the constant supplied by Lemma~\ref{lem:boundary-propagation}. We consider two cases, depending on whether there exists a rank $q\in\{Q,\ldots,r\}$ such that $\mu_q<\rho$. 

\medskip\noindent\textbf{Case 1: $\mu_q<\rho$ for some
$q\in\{Q,\ldots,r\}$.}
We define
\[
        k\coloneqq \max\big\{q\in\{Q,\ldots,r\}:\mu_q<\rho\big\}.
\]
By the maximality of $k$, $\mu_q\ge\rho$ for every
$q\in\{k+1,\ldots,r\}$.  Lemma~\ref{lem:boundary-propagation} therefore
gives
\[
        D_k k^{p-1}\le K_pB r^{p-1}.
\]
Combining this with \eqref{eq:small-matching-product} yields
\[
\begin{split}
 \cC_T(h)
 &>\frac{\rho^{-k}}{2D_k}
 =\frac{k^{p-1}\rho^{-k}}{2D_k k^{p-1}}
 \ge\frac{Q^{p-1}\rho^{-Q}}
 {2K_pB r^{p-1}}
 \ge\frac{Q^{p-1}\rho^{-Q}}
 {2K_pB(r+1)^{p-1}}.
\end{split}
\]

\medskip\noindent\textbf{Case 2: $\mu_q\ge\rho$ for every
$q\in\{Q,\ldots,r\}$.}
Set $q_0:=\min\{r,Q\}$. If $r<Q$, then $q_0=r$, and hence $D_{q_0}=D_r=B$. If $r\ge Q$, then $q_0=Q$; moreover, the assumption of this case gives $\mu_q\ge\rho$ for all $q\in\{Q,\ldots,r\}$. Applying Lemma~\ref{lem:boundary-propagation} with $k=Q$ therefore yields \[ D_{q_0}=D_Q \le K_pB\left(\frac rQ\right)^{p-1}. \] Since $K_p\ge1$ and $Q\ge1$, these two subcases are both covered by the uniform bound \[ D_{q_0}\le K_pB(r+1)^{p-1}. \] Lemma~\ref{lem:bounded-rank} then gives \[ \cC_T(h) \ge\frac{c_Q}{D_{q_0}} \ge\frac{c_Q/K_p}{B(r+1)^{p-1}}. \] 

\paragraph{Putting the two cases together.}
Note that the quantities $\rho$, $Q$,
$K_p$, and $c_Q$ depend only on $p$.  We can therefore define
\[
        c_p:=
        \min\left\{
        \frac{Q^{p-1}\rho^{-Q}}{2K_p},
        \frac{c_Q}{K_p}
        \right\}>0.
\]
With this choice of $c_p$, \eqref{eq:normalized-lower-bound} follows from the two cases above.

\subsection{Restoring the original scaling}
\label{sec:restoration}

Thus far, we have established a lower bound on the normalized
quantity $\cC_T(h)$.  By Corollary~\ref{cor:functional-attainment}, this
value is attained by a finite-dimensional smooth convex instance.   It remains only to restore the original smoothness parameter $L$ and initial-distance scale $R$. To this end, we have the following lemma. 

\begin{lemma}
\label{lem:scaling}
Let $F\in\cF_{0,1}(\R^d)$, let $L,R>0$, and define
\[
        f(x):=LR^2F(x/R).
\]
Then $f\in\cF_{0,L}(\R^d)$ and
$\nabla f(x)=LR\nabla F(x/R)$.  If $h_t=L\eta_t$ and
\[
 \bar x_{t+1}=\bar x_t-h_t\nabla F(\bar x_t),
 \qquad x_t:=R\bar x_t,
\]
then $x_{t+1}=x_t-\eta_t\nabla f(x_t)$.  Moreover, every
$\bar x_\star\in\argmin F$ gives $x_\star:=R\bar x_\star\in\argmin f$,
and
\[
 \|x_0-x_\star\|=R\|\bar x_0-\bar x_\star\|,
 \qquad
 f(x_T)-f(x_\star)
 =LR^2\{F(\bar x_T)-F(\bar x_\star)\}.
\]
\end{lemma}

\begin{proof}
We first note that positive scaling and composition with $x\mapsto x/R$
preserve convexity and carry minimizers to $R\argmin F$.  The chain rule
then gives the stated gradient formula.  Moreover, $1$-smoothness of $F$
yields
\[
 \|\nabla f(x)-\nabla f(y)\|
 \le LR\|(x-y)/R\|=L\|x-y\|.
\]
Finally, it is readily seen that
\[
 x_t-\eta_t\nabla f(x_t)
 =R\{\bar x_t-h_t\nabla F(\bar x_t)\}
 =R\bar x_{t+1}.
\]
The distance and objective gap identities follow directly from the
definitions.
\end{proof}

\begin{proof}[Proof of Theorem~\ref{thm:main}]
Fix $p\in(p_\star,2)$, an integer $T\ge1$, parameters $L,R>0$, and a stepsize
schedule $\eta\in\R_{\ge0}^T$.  We then set $h_t=L\eta_t$ and choose
$c_p>0$ small enough that Proposition~\ref{prop:normalized-floor}
holds with $2c_p$ in place of its constant.

Since $B\le T+1$ and $r+1\le T+1$,
Proposition~\ref{prop:normalized-floor} gives
\[
 \cC_T(h)
 \ge \frac{2c_p}
          {B(r+1)^{p-1}}
 \ge 2c_p(T+1)^{-p}.
\]
By Corollary~\ref{cor:functional-attainment}, there exist
an integer $\bar d$ with $1\le\bar d\le T+1$, a function
$F\in\cF_{0,1}(\R^{\bar d})$ with $0\in\argmin F$, and a normalized GD
trajectory satisfying
\[
 \bar x_{t+1}=\bar x_t-h_t\nabla F(\bar x_t)
 \quad(0\le t<T),
 \qquad \|\bar x_0\|=1,
\]
for which
\[
 2\{F(\bar x_T)-F(0)\}
 =\cC_T(h)
 \ge 2c_p(T+1)^{-p}.
\]

Apply Lemma~\ref{lem:scaling} with $\bar x_\star=0$ and define
\[
 f(x)=LR^2F(x/R),\qquad x_t=R\bar x_t,
 \qquad x_\star=0.
\]
Then $f\in\cF_{0,L}(\R^{\bar d})$, the sequence $(x_t)_{t=0}^T$ is exactly
the GD trajectory with stepsizes $\eta_t$, and $\|x_0-x_\star\|=R$.  Moreover,
\[
 f(x_T)-f(x_\star)
 \ge c_pLR^2(T+1)^{-p}.
\]
This concludes the proof of  Theorem~\ref{thm:main}.
\end{proof}

\section{Discussion}
\label{sec:discussion}

In this work, we have established a new lower bound for the last-iterate convergence rate
of plain GD with arbitrary predetermined nonnegative
stepsize schedules designed for a prescribed horizon. In particular, our
result has provided the first rigorous evidence that stepsize schedules alone
cannot accelerate plain GD to the optimal $O(T^{-2})$
convergence rate.  The tight worst-case rate achievable by this class of
stepsize-based acceleration methods, however, remains open. We conclude by highlighting several
directions for future work.

\begin{itemize}
\item \textbf{Determining the optimal convergence exponent.}
As mentioned previously, the best-known upper bound for stepsize-based acceleration of GD is
$O(T^{-\log_2(1+\sqrt2)})=O(T^{-1.2715\ldots})$, whereas our lower bound
rules out polynomial exponents exceeding
$p_\star\approx 1.9319$. Closing this gap by matching the upper and lower bounds remains a
central open problem.

\item \textbf{Lower bounds for the best iterate.}
Thus far, Theorem~\ref{thm:main} concerns only the last iterate $x_T$. A
stronger result would establish a lower bound for
\[
\min_{0\le t\le T}\big\{f(x_t)-f(x_\star)\big\}.
\]
Such a result cannot be obtained by applying
Theorem~\ref{thm:main} independently to each prefix, since different
prefixes may require different hard instances. 

\item \textbf{Lower bounds for strongly convex problems.}
The silver stepsize schedule also accelerates gradient descent in the strongly convex
setting \citep{altschuler2025acceleration}. An analogous lower bound for
predetermined nonnegative stepsize schedules would improve our
understanding of the optimal convergence rate in the strongly convex
regime. 

\item \textbf{Allowing negative or adaptive stepsizes.}
Our lower bound relies on both the nonnegativity of the stepsizes and the
fact that the entire schedule is fixed in advance. Whether analogous
lower bounds hold for signed stepsizes or adaptively chosen stepsizes
remains an interesting question for future exploration.
\end{itemize}

\section*{Disclosure on the use of generative AI}

The main proof was developed by GPT-5.6 Sol Pro. The authors provided the model with two inputs: the research objective of proving a lower bound for plain gradient descent showing that its worst-case rate is strictly slower than $1/T^2$, and a high-level resisting-oracle strategy. That strategy was to
first construct an adversarial gradient descent trajectory and then show that
the trajectory could be realized by a smooth convex function. Beyond this
objective and high-level strategy, the authors did not provide any nontrivial
mathematical ingredients used in the final proof.

The proof was not produced in a single interaction. The authors queried GPT-5.6 Sol Pro multiple times, and the argument presented in this paper emerged only after several attempts. GPT-5.6 Sol was also used in later rounds to improve the organization and exposition of the proof. The authors spent substantial effort reviewing, verifying, and revising the generated material to make the proof correct and readable. The authors additionally used Codex to formalize the proof in Lean 4. The formalization is available at \url{https://github.com/jianhaoma/gd-lower-bound-lean}. The authors take full responsibility for the final manuscript and all of its mathematical claims.

\section*{Acknowledgement}
Jianhao thanks Stephen Wright for helpful discussions during the preparation of this manuscript.

\appendix
\section{Proof of the projection formula for Moreau support envelopes}
\label{app:moreau-projection}

In this section, we prove Lemma~\ref{lem:moreau-projection}.

\begin{proof}[Proof of Lemma~\ref{lem:moreau-projection}]
We fix $x\in\R^d$ and set
\[
        x_K\coloneqq\Pi_K(x),
        \qquad z\coloneqq x-x_K.
\]
The variational
characterization of Euclidean projection gives
\[
        \langle z,v-x_K\rangle\le0
        \qquad \text{for every }v\in K.
\]
Thus, $x_K$ maximizes
$\langle\cdot,z\rangle$ over $K$.  Since
$\partial\sigma_K(z)
=\{g\in K:\langle g,z\rangle=\sigma_K(z)\}$, we have
$x_K\in\partial\sigma_K(z)$.  Together with $z-x=-x_K$, this yields
\[
        0\in\partial\sigma_K(z)+z-x.
\]
Hence, $z$ is the unique minimizer of the problem in \eqref{eq:proximal-point}, because
its objective is strongly convex.  It follows that
\[
        \operatorname{prox}_{\sigma_K}(x)
        =z=x-\Pi_K(x),
\]
which proves \eqref{eq:moreau-proximal-formula}.

The support equality $\sigma_K(z)=\langle x_K,z\rangle$ now gives
\begin{align}
 F(x)
 &=\langle x_K,z\rangle+\frac12\|x_K\|^2
  =\langle x_K,x\rangle-\frac12\|x_K\|^2 \notag\\
 &=\frac12\|x\|^2-\frac12\|x-x_K\|^2
  =\frac12\|x\|^2-\frac12\dist(x,K)^2.
 \label{eq:moreau-value-proof}
\end{align}
Completing the square gives the equivalent representation
\begin{equation}
 F(x)
 =\max_{g\in K}
   \left\{\langle g,x\rangle-\frac12\|g\|^2\right\}.
 \label{eq:moreau-dual-representation}
\end{equation}
Completing the square also shows that the maximizer in
\eqref{eq:moreau-dual-representation} is $x_K$, and strong concavity ensures
that this maximizer is unique.  Moreover, the same representation
shows that $F$ is a pointwise maximum of affine functions of $x$ and is
therefore convex.

To identify the gradient, we fix $y\in\R^d$ and set
\[
        y_K\coloneqq\Pi_K(y),
        \qquad \Delta\coloneqq y-x.
\]
Using the maximization formula for $F(\cdot)$, we have
\begin{align*}
    F(y) &\geq \langle x_K, y \rangle - \frac{1}{2} \| x_K\|^2 = F(x) + \langle x_K, \Delta \rangle, \\
    F(x) &\geq \langle y_K, x \rangle - \frac{1}{2} \| y_K\|^2 = F(y) - \langle y_K, \Delta \rangle.
\end{align*}
Combining these inequalities gives
\[
        \langle x_K,\Delta\rangle
        \le F(y)-F(x)
        \le\langle y_K,\Delta\rangle.
\]
Consequently,
\[
 0\le F(y)-F(x)-\langle x_K,\Delta\rangle
 \le\langle y_K-x_K,\Delta\rangle
 \le\|y_K-x_K\|\,\|\Delta\|
 \le\|\Delta\|^2,
\]
where the last inequality uses the nonexpansiveness of Euclidean
projection.  Dividing by $\|\Delta\|$ and letting $y\to x$ shows that $F$ is
Fr\'echet differentiable at $x$ with
\[
        \nabla F(x)=x_K=\Pi_K(x).
\]
Because $x$ was arbitrary, the nonexpansiveness of $\Pi_K$ also shows that
$\nabla F$ is $1$-Lipschitz.  Together with
\eqref{eq:moreau-value-proof}, this proves
\eqref{eq:moreau-projection-formula} and completes the proof.

\end{proof}

\section{Technical proofs for the schedule analysis}
\label{app:schedule-proofs}

We collect here four estimates used in the schedule analysis: the
equalization bound under a mass constraint, the recursions as the rank
cutoff changes, the one-step Lyapunov estimate, and the bounded-rank
argument.

\subsection{Proof of Lemma~\ref{lem:budget-equalization}}
\label{app:budget-equalization}

Set $\widehat u_i \coloneqq u_i/\bar u$ and $\beta_i \coloneqq \bar u B_i/A_i$.  For
$x>0$ and $\beta\ge0$,
\[
 \frac{x(1+\beta x)}{1+2\beta}
 =\frac{x}{1+2\beta}
  +\frac{2\beta}{1+2\beta}\frac{x^2}{2}
 \le e^{x-1}.
\]
Indeed, $x\le e^{x-1}$, and $x^2/2\le e^{x-1}$ because
$g(x):=x^2e^{1-x}/2$ satisfies $g'(x)=g(x)(2/x-1)$ and hence attains its maximum at $x=2$, with
$g(2)=2/e<1$.  The ratio of the left-hand side of
\eqref{eq:budget-product} to its right-hand side is therefore
\[
 \frac{
   (\prod_{i=1}^q u_i)\prod_{i=1}^{q-1}(A_i+B_iu_i)}
  {\bar u^q\prod_{i=1}^{q-1}(A_i+2\bar u B_i)}
 =\widehat u_q\prod_{i=1}^{q-1}
 \frac{\widehat u_i(1+\beta_i\widehat u_i)}{1+2\beta_i}
 \le \exp\!\left(\sum_{i=1}^q\widehat u_i-q\right)
 \le1,
\]
since $\sum_i\widehat u_i=\bar u^{-1}\sum_i u_i\le D/\bar u=q$.  The same
conclusion holds when $q=1$, in which case the product over $i<q$ is empty.

\subsection{Proof of Lemma~\ref{lem:adjacent-rank-relations}}
\label{app:adjacent-rank-relations}

The identity $D_{q-1}=D_q+a_q$ and the definition of $\nu_q$
give \eqref{eq:one-rank-mass-ratio}.  Multiplying these adjacent ratios and
using $D_r=B$ lead to \eqref{eq:mass-product}, with an empty product when
$k=r$.

Next, we divide $a_q\ge a_{q+1}$ by $D_q$ and use
$D_q=D_{q+1}\{1+\nu_{q+1}/(q+1)\}$ to obtain
\[
 \frac{\nu_q}{q}
 =\frac{a_q}{D_q}
 \ge\frac{a_{q+1}}{D_q}
 =\frac{\nu_{q+1}/(q+1)}
        {1+\nu_{q+1}/(q+1)}.
\]
This is exactly \eqref{eq:mass-increment-transition}.  Rearranging terms gives
\[
        \nu_{q+1}(q-\nu_q)\le(q+1)\nu_q.
\]
When $q>\nu_q$, division by the positive denominator proves
\eqref{eq:mass-increment-bound}.

Finally, the definition of $\zeta_q$ gives
\[
 \sum_{s=1}^{q+1}\frac1{a_s}
 =\frac{q^2\zeta_q}{D_q}
  +\frac{q+1}{\nu_{q+1}D_{q+1}}.
\]
Therefore, we obtain
\[
\begin{split}
 \zeta_{q+1}
 &=\frac{D_{q+1}}{(q+1)^2}
   \left(\frac{q^2\zeta_q}{D_q}
   +\frac{q+1}{\nu_{q+1}D_{q+1}}\right)\\
 &=\frac{q^2\zeta_q}{(q+1)^2}\frac{D_{q+1}}{D_q}
   +\frac1{\nu_{q+1}(q+1)}\\
 &=\frac{q^2\zeta_q}{(q+1)(q+1+\nu_{q+1})}
   +\frac1{\nu_{q+1}(q+1)},
\end{split}
\]
which establishes \eqref{eq:zeta-recursion}.

\subsection{Proof of Lemma~\ref{lem:one-step-lyapunov}}
\label{app:one-step-lyapunov-proof}

In this proof, we use the recursions and potential from
Subsection~\ref{subsec:mass-growth}. 
We shall also use the notation
\[
        A(v):=\frac{v}{\vartheta(v+p+1)},
        \qquad
        \mathcal L_q=A(\nu_q)(\zeta_q-\vartheta), 
\]
and choose $C_p\ge1$ so that
\[
 C_p\ge
 \frac{\vartheta^{-1}
       \{1+\vartheta^{-1}(\vartheta^{-1}+2)\}}{p+1}.
\]

\medskip\noindent\textbf{Coefficient comparison.}
We abbreviate $v:=\nu_n$ and $u:=\nu_{n-1}$, and we define
\[
        \omega_n(v):=\frac{(n-1)^2}{n(n+v)}.
\]
Because $n-1>\vartheta^{-1}\ge u$, the denominator in the
assumed upper bound on $v$ is positive.  We now set
\[
        u_{\min}:=\frac{(n-1)v}{n+v}.
\]
Rearranging that bound gives
\begin{equation}
        u\ge u_{\min}.
 \label{eq:reset-lower}
\end{equation}
The derivative of $A$ and the ratio $A(x)/x$ satisfy
\[
 A'(x)=\frac{p+1}{\vartheta(x+p+1)^2}>0,
 \qquad
 \frac{A(x)}{x}=\frac1{\vartheta(x+p+1)}.
\]
Thus, $A$ is increasing on $(0,\infty)$, whereas $A(x)/x$ is
decreasing there.  We also have $0<u_{\min}<v$, so
\begin{equation}
\begin{aligned}
 A(u)
 &\ge A(u_{\min})
 \ge \frac{u_{\min}}{v}A(v)
 =\frac{n-1}{n+v}A(v)\ge \frac{(n-1)^2}{n(n+v)}A(v)
 =\omega_n(v)A(v).
\end{aligned}
 \label{eq:lyapunov-coefficient}
\end{equation}

\medskip\noindent\textbf{Drift calculation.}
We define the drift term
\[
 \Delta_n(v):=\frac1{nv}
  -\frac{\vartheta\{n(v+2)-1\}}{n(n+v)}.
\]
Subtracting $\vartheta$ from the recursion gives
\[
 \zeta_n-\vartheta
 =\omega_n(v)(\zeta_{n-1}-\vartheta)+\Delta_n(v).
\]
Here we use the identity
\[
 1-\omega_n(v)
 =\frac{n(n+v)-(n-1)^2}{n(n+v)}
 =\frac{n(v+2)-1}{n(n+v)}.
\]
Using \eqref{eq:parameter-choice}, direct simplification gives
\begin{equation}
 A(v)\Delta_n(v)
 =\frac1n\left[
   p-1-v+
   \frac{v\{1+v(v+2)\}}{(n+v)(v+p+1)}
 \right].
 \label{eq:lyapunov-boundary}
\end{equation}
To verify this identity, we use
$\vartheta^{-1}=(p-1)(p+1)$ and $(p+1)-(p-1)=2$ to obtain
\[
\begin{split}
 nA(v)\Delta_n(v)-(p-1-v)
 &=\frac{v}{v+p+1}
 \left[(v+2)-\frac{n(v+2)-1}{n+v}\right]
 =\frac{v\{1+v(v+2)\}}{(n+v)(v+p+1)},
\end{split}
\]
which is precisely \eqref{eq:lyapunov-boundary}.
Because $\zeta_{n-1}-\vartheta\ge0$,
\eqref{eq:lyapunov-coefficient} gives
\[
\begin{split}
 \mathcal L_n
 &=A(v)\omega_n(v)(\zeta_{n-1}-\vartheta)+A(v)\Delta_n(v)\le A(u)(\zeta_{n-1}-\vartheta)+A(v)\Delta_n(v)=\mathcal L_{n-1}+A(v)\Delta_n(v).
\end{split}
\]
Moreover, $0<v\le\vartheta^{-1}$ implies
\[
 0\le
 \frac{v\{1+v(v+2)\}}{(n+v)(v+p+1)}
 \le\frac1n
 \frac{\vartheta^{-1}
       \{1+\vartheta^{-1}(\vartheta^{-1}+2)\}}{p+1}
 \le\frac{C_p}{n}.
\]
Here we use $n/(n+v)\le1$, $v/(v+p+1)\le
\vartheta^{-1}/(p+1)$, and
$1+v(v+2)\le1+\vartheta^{-1}(\vartheta^{-1}+2)$.
Combining this estimate with \eqref{eq:lyapunov-boundary} proves
\eqref{eq:one-step-lyapunov}.

\subsection{The bounded-rank case}
\label{app:bounded-rank}

The next lemma provides the estimate used in the proof of
Lemma~\ref{lem:bounded-rank}.

\begin{lemma}[Bounded-rank prefix bound]
\label{lem:bounded-rank-prefix}
Let $0\le q_0\le r$, and regard $D_{q_0}$ as the residual schedule mass after the
top $q_0$ excesses have been selected.  When $q_0\ge1$, consider the
chronological chains whose selected sets are the global rank prefixes
\[
 \varnothing,\quad
 \{\tau_1\},\quad
 \{\tau_1,\tau_2\},\quad
 \ldots,\quad
 \{\tau_1,\ldots,\tau_{q_0}\}.
\]
Each nonempty prefix is restored to chronological order before applying the
geometric construction, with equality in each local bound.  The value of a
chain with $m$ selected steps is
 \[
        \frac1{H_m}\prod_{i=0}^{m-1}\chi_i.
 \]
This is the chain's contribution to $\cC_T(h)$ in
\eqref{eq:lower-bound-functional}.  When $q_0\ge1$, at least one of these chains
has value at least
\begin{equation}
        \frac1{4D_{q_0}\,q_0\,8^{q_0-1}}.
 \label{eq:bounded-rank-prefix}
\end{equation}
For $q_0=0$, the empty chain has value at least $(2D_0)^{-1}$.
\end{lemma}

\begin{proof}[Proof of Lemma~\ref{lem:bounded-rank-prefix}]
We use a diffuse--dense dichotomy.  Suppose first that, for every prefix, the
product of its smallest excess and its length is smaller than the residual schedule mass behind it.  Telescoping then shows that the top $q_0$ excesses increase
the residual schedule mass $D_{q_0}$ by at most a factor $q_0+1$, and the empty chain
gives the desired bound.  If a dense prefix exists instead, the largest such
prefix has both a controlled residual schedule mass and a large smallest selected
excess.  These two facts control the local factors in the construction.

If $q_0=0$, the empty-chain value is
$(2D_0-1)^{-1}\ge(2D_0)^{-1}$.  For the rest of the proof, we assume that
$q_0\ge1$.
For $0\le j\le q_0$, the residual schedule masses satisfy
\[
 D_j=D_{q_0}+\sum_{s=j+1}^{q_0}a_s,
 \qquad
 D_{j-1}=D_j+a_j.
\]
We call the rank prefix of length $j\in\{1,\ldots,q_0\}$ \emph{dense} if
$j a_j\ge D_j$.

\medskip\noindent\textbf{Case 1: all prefixes are diffuse.}
In this case, $j a_j<D_j$ for every $1\le j\le q_0$.  The tail recursion
telescopes to
\[
 \frac{D_0}{D_{q_0}}
 =\prod_{j=1}^{q_0}
   \left(1+\frac{a_j}{D_j}\right)
 <\prod_{j=1}^{q_0}\left(1+\frac1j\right)
 =q_0+1.
\]
Consequently, the empty-chain value satisfies
\[
 \frac1{2D_0-1}
 \ge\frac1{2D_0}
 >\frac1{2D_{q_0}(q_0+1)}
 \ge\frac1{4D_{q_0}\,q_0\,8^{q_0-1}}.
\]
The last inequality follows from
$2(q_0+1)\le4q_0\,8^{q_0-1}$.

\medskip\noindent\textbf{Case 2: a dense prefix exists.}
We now take $q$ to be the largest index for which the rank prefix is dense. Every prefix of length $j>q$ is diffuse, so
telescoping the residual schedule mass beyond $q$ gives
\[
 \frac{D_q}{D_{q_0}}
 =\prod_{j=q+1}^{q_0}
   \left(1+\frac{a_j}{D_j}\right)
 \le\prod_{j=q+1}^{q_0}\left(1+\frac1j\right)
 =\frac{q_0+1}{q+1}.
\]
Together with the density of the prefix of length $q$, this gives
\begin{equation}
 D_q\le D_{q_0}\frac{q_0+1}{q+1},
 \qquad
 a_q\ge\frac{D_q}{q}.
 \label{eq:bounded-rank-dense-consequences}
\end{equation}
The first inequality controls the residual schedule mass after the prefix, while the
second controls every $\chi_i$ through the smallest selected excess.

We select the labeled prefix $\{\tau_1,\ldots,\tau_q\}$ and restore its
chronological order.  Specifically, we define $\pi$ as the unique permutation
of $\{1,\ldots,q\}$ such that
\[
 \tau_{\pi(1)}<\cdots<\tau_{\pi(q)},
\]
and we then set $c_i:=a_{\pi(i)}=y_{\tau_{\pi(i)}}$ and use the block notation
with $t_i:=\tau_{\pi(i)}$ for $1\le i\le q$.  We also set
\[
 V:=1+\sum_{t=\tau_{\pi(q)}+1}^{T-1}h_t,
 \qquad
 H_q=2V-1.
\]
Because the gaps partition the time indices outside the selected prefix,
\begin{equation}
\begin{split}
 \sum_{i=0}^{q-1}U_i+V
 &=q+1+
   \sum_{t\notin\{\tau_1,\ldots,\tau_q\}}h_t=B+\sum_{s=q+1}^{r}a_s
 =D_q.
\end{split}
 \label{eq:bounded-rank-mass-decomposition}
\end{equation}
For this chain, we set $\gamma_i^2=\chi_i$ for $0\le i<q$ and define
$a_{\min}:=a_q$.

\medskip\noindent\textbf{The nonterminal factors.}
We define $\varphi(x):=2x+x^2$.  For $1\le i<q$,
$c_i,c_{i+1}\ge a_{\min}$ and
$H_i=U_i+c_{i+1}\ge a_{\min}$.  Hence, for $1\le i<q$, the first identity in
\eqref{eq:local-factor-identities} gives
\begin{equation}
        \chi_{i-1}^{-1}
        \le\varphi(U_{i-1}/a_{\min}).
 \label{eq:bounded-rank-internal}
\end{equation}
Indeed, the three terms in that identity are bounded by
$U_{i-1}/a_{\min}$, $U_{i-1}/a_{\min}$, and
$U_{i-1}^2/a_{\min}^2$.  Moreover,
\[
        \sum_{i=0}^{q-2}\frac{U_i}{a_{\min}}
        \le\frac{D_q}{a_{\min}}\le q
\]
by \eqref{eq:bounded-rank-mass-decomposition} and
\eqref{eq:bounded-rank-dense-consequences}.  Since
\[
        \frac{d^2}{dx^2}\log\varphi(x)
        =-\frac1{x^2}-\frac1{(x+2)^2}<0,
\]
$\log\varphi$ is concave on $(0,\infty)$.
Moreover, $\varphi'(x)=2+2x>0$, so $\varphi$ is increasing.  For $q\ge2$,
Jensen's inequality and the preceding mass bound give
\begin{equation}
\begin{split}
 \prod_{i=0}^{q-2}\chi_i^{-1}
 &\le
 \varphi\!\left(
   \frac1{q-1}\sum_{i=0}^{q-2}\frac{U_i}{a_{\min}}
 \right)^{q-1}\le\varphi\!\left(\frac{q}{q-1}\right)^{q-1}
 \le8^{q-1}.
\end{split}
 \label{eq:bounded-rank-internal-product}
\end{equation}
The last inequality uses $q/(q-1)\le2$ and $\varphi(2)=8$.  When $q=1$,
the internal product is empty and the same bound holds.  Thus
\begin{equation}
        \prod_{i=0}^{q-2}\chi_i^{-1}\le8^{q-1}
        \qquad(q\ge1).
 \label{eq:bounded-rank-internal-summary}
\end{equation}

\medskip\noindent\textbf{The terminal factor.}
The mass decomposition \eqref{eq:bounded-rank-mass-decomposition} also gives,
with the sum interpreted as empty when $q=1$,
\[
 2D_q-(U_{q-1}+H_q)
 =2\sum_{i=0}^{q-2}U_i+U_{q-1}+1\ge0.
\]
Thus $U_{q-1}\le D_q$ and
$U_{q-1}+H_q\le2D_q$.  The second identity in
\eqref{eq:local-factor-identities}, followed by
\eqref{eq:bounded-rank-dense-consequences}, now yields
\begin{equation}
\begin{split}
 H_q\chi_{q-1}^{-1}
 &\le D_q
       \left(1+\frac{2D_q}{a_{\min}}\right)
 \le3qD_q
 \le3D_{q_0}\,q_0.
\end{split}
 \label{eq:bounded-rank-terminal}
\end{equation}
Here the last inequality uses
$q(q_0+1)/(q+1)\le q_0$.

The selected prefix therefore has value
\[
 \frac1{H_q}\prod_{i=0}^{q-1}\chi_i
 =\frac1{(H_q\chi_{q-1}^{-1})
          \prod_{i=0}^{q-2}\chi_i^{-1}}
 \ge\frac1{3D_{q_0}\,q_0\,8^{q-1}}
 \ge\frac1{4D_{q_0}\,q_0\,8^{q_0-1}}.
\]
We state the result with $4$ and $q_0$, rather than $3$ and the selected
rank $q$, to obtain a simple bound uniform over all possible prefixes.
This proves \eqref{eq:bounded-rank-prefix}.
\end{proof}

\begin{proof}[Proof of Lemma~\ref{lem:bounded-rank}]
Lemma~\ref{lem:bounded-rank-prefix} gives
$\cC_T(h)\ge(2D_0)^{-1}$ when $q_0=0$, and for
$1\le q_0\le\min\{Q,r\}$ it gives
\[
 \cC_T(h)
 \ge\frac1{4D_{q_0}\,q_0\,8^{q_0-1}}.
\]
We can therefore set
\[
        c_Q:=\frac1{4Q\,8^{Q-1}}.
\]
This choice works simultaneously for every $0\le q_0\le\min\{Q,r\}$.
Indeed, $c_Q\le1/2$, and
$q_0 8^{q_0-1}\le Q8^{Q-1}$ for $1\le q_0\le Q$.
\end{proof}

\bibliographystyle{plainnat}
\bibliography{references}

\end{document}